\documentclass[12pt]{article}
\usepackage{amsmath, amssymb, amsfonts, amsthm, epsfig, graphicx}
\usepackage[all]{xy}
\title{New Uniform Diameter Bounds in Pro-$p$ Groups}
\author{Henry Bradford \\ University of Oxford}
\date{}

\begin{document}

\newtheorem{thm}{Theorem}[section]
\newtheorem{lem}[thm]{Lemma}
\newtheorem{propn}[thm]{Proposition}
\newtheorem{coroll}[thm]{Corollary}
\newtheorem{defn}[thm]{Definition}
\newtheorem{ex}[thm]{Example}
\newtheorem{notn}[thm]{Notation}
\newtheorem{conj}[thm]{Conjecture}
\newtheorem{note}[thm]{Note}
\newtheorem{rmrk}[thm]{Remark}

\maketitle

\begin{abstract}
We give new upper bounds for the diameters of finite groups which do not depend on a choice of generating set. 
Our method exploits the commutator structure of certain profinite groups, 
in a fashion analogous to the Solovay-Kitaev procedure from quantum computation. 
We obtain polylogarithmic upper bounds for the diameters of finite quotients of: 
groups with an analytic structure over a pro-$p$ domain (with exponent depending on the dimension); 
Chevalley groups over a pro-$p$ domain (with exponent independent of the dimension) 
and the Nottingham group of a finite field. 
We also discuss some consequences of our results for random walks on groups. 
\end{abstract}

\section{Introduction}

The interplay between growth, spectral gap and diameter for finite groups has been a highly active area of study within group theory in recent years. Given a finite group $G$ and a generating set $S \subseteq G$, 
recall that the \emph{diameter} of the pair $(G,S)$ is given by:
\begin{center}
$diam(G,S) = min \lbrace n \in \mathbb{N} : B_S (n) = G \rbrace$
\end{center} 
where $B_S (n)$ is the (closed) \emph{word-ball} of radius $n$, given by:
\begin{center}
$B_S (n) = \lbrace s_1 \cdots s_n : s_1 , \ldots , s_n \in S \cup S^{-1} \cup \lbrace 1 \rbrace \rbrace$. 
\end{center}
In this paper we investigate techniques for establishing upper bounds for $diam(G,S)$ 
(such bounds will usually be expressed as a function of $\lvert G \rvert$). 

Meanwhile the \emph{spectral gap} of the pair $(G,S)$ is a measure of the \emph{mixing time} 
of the simple random walk on $(G,S)$: that is, if $(G,S)$ 
has large spectral gap then only a small number of steps  must be taken in such a walk before 
the associated probability distribution on $G$ is close to uniform 
(we make these notions and their relationship to diameter precise later). 

Our interest shall be in upper bounds for the diameter which do not depend on the generators. 
With this in mind, we define for a finite group $G$:
\begin{center}
$diam(G) = max \lbrace diam(G,S) : S \subseteq G , \langle S \rangle = G \rbrace$. 
\end{center}
This quantity is referred to by many authors as the \emph{worst-case diameter} for $G$.

\subsection{Statement of Results}

Fix $p$ prime. Let $R$ be a commutative unital Noetherian ring. 
Recall that $R$ is called a \emph{local ring} if $R$ has a unique non-zero maximal ideal $\mathcal{M}$ 
(we shall refer to \emph{the local ring $(R,\mathcal{M})$}). 
The quotient $R / \mathcal{M}$ is called the \emph{residue field} of $R$. 
There is a topology on $R$, called the \emph{$\mathcal{M}$-adic topology}, 
induced by declaring the filtration $(\mathcal{M}^n)_n$ to be a basis for the neighbourhoods of $0$. 

\begin{defn}
The local ring $(R,\mathcal{M})$ is called a \emph{pro-$p$ ring} if: 
\begin{itemize}
\item[(i)] The residue field of $R$ is finite of characteristic $p$ 
\item[(ii)] $R$ is complete with respect to the $\mathcal{M}$-adic topology. 
\end{itemize}
A pro-$p$ ring $(R,\mathcal{M})$ where $\mathcal{M}$ is principal is called a \emph{discrete valuation pro-$p$ ring}
and a pro-$p$ ring which is an integral domain will be called a \emph{pro-$p$ domain}. 
\end{defn}

Let $\mathbb{K}$ be the field of fractions of $R$. 
Fix $c \in (0,1)$ and define a norm $\lVert \cdot \rVert$ on $R$ 
(compatible with the $\mathcal{M}$-adic topology) by:
\begin{center}
$\lVert a \rVert = c^n$ for $a \in \mathcal{M}^n \setminus \mathcal{M}^{n+1}$; $\lVert 0 \rVert = 0$. 
\end{center}
In particular if $(R,\mathcal{M})$ is a discrete valuation ring, 
with $\mathcal{P} \in \mathcal{M}$ such that $\mathcal{M} = (\mathcal{P})$,
then $\lVert \mathcal{P} \rVert = c$. In this case, we extend $\Vert \cdot \rVert$ to $\mathbb{K}$ via:
\begin{center}
$\lVert a \rVert = \lVert a \mathcal{P}^n \rVert c^{-n}$ for $n$ sufficiently large that $a \mathcal{P}^n \in R$. 
\end{center}
In what follows we take $(R,\mathcal{M})$ to be a pro-$p$ domain and a discrete valuation ring. 
By work of Cohen \cite{Cohen} every such $R$ arises as a finitely generated free module over 
a subring of the form $\mathbb{Z}_p$ or $\mathbb{F}_p [[t]]$. 

Our first result concerns compact groups with a compatible structure as an $R$-analytic manifold. 
Every such group has an open subgroup with an especially simple $R$-analytic structure, 
called an \emph{$R$-standard group}. Precise definitions are given in Section \ref{AnalyticSection}. 
In a $d$-dimensional $R$-analytic group $G$, an $R$-standard subgroup may be identified 
(as a space) with the product space $\mathcal{M}^{(d)}$; the balls $(\mathcal{M}^n)^{(d)}$ around $0$ 
form a filtration by open normal subgroups. 
The commutator structure in $\mathcal{M}^{(d)}$ is controlled by the Lie algebra $\mathcal{L}_G$. 
Recall that a Lie algebra $\mathcal{L}$ is called \emph{perfect} if $\mathcal{L}$ is equal to 
its derived subalgebra $(\mathcal{L},\mathcal{L})$. 

\begin{thm} \label{Rstandarddiam}
Let $G$ be a $d$-dimensional $R$-standard group, \linebreak $K_n = (\mathcal{M}^n)^{(d)} \vartriangleleft_o G$. 
Suppose $\mathcal{L}_{G}$ is perfect. Then there exist $C_1 (G)$, $C_2 (d) > 0$ such that:
\begin{center}
$diam(G / K_n) \leq C_1 (log \lvert G/K_n \rvert)^{C_2}$. 
\end{center}
\end{thm}

In the case $R = \mathbb{Z}_p$, we can say more, exploiting the concept of a \emph{uniform} subgroup 
and associated additional features of the Lie theory (explained in detail in Section \ref{padicanalyticSection}). 
Every $\mathbb{Z}_p$-standard group is uniform 
and every compact $p$-adic analytic group has an open characteristic uniform subgroup. 
In a $\mathbb{Z}_p$-standard group $G$, the balls $K_n$ described in Theorem \ref{Rstandarddiam} 
coincide with the terms of the lower central $p$-series for $G$. 

\begin{defn}
Let $G$ be a profinite group. $G$ is \emph{FAb} if every open subgroup has finite abelianisation. 
\end{defn}

\begin{thm} \label{padicdiam}
Let $p \geq 3$. Let $G$ be a $d$-dimensional compact $p$-adic analytic group. 
Let $K_1 \leq G$ be an open characteristic uniform subgroup; 
$(K_n)_n$ its lower central $p$-series. 
If $G$ is FAb then there exist $C_1 (G) , C_2 (d) > 0$ such that: 
\begin{equation} \label{padicdiamineq}
diam(G / K_n) \leq C_1 (log \lvert G / K_n \rvert)^{C_2}. 
\end{equation}
For $G=K_1$ then conversely: if $G$ satisfies (\ref{padicdiamineq}) then $G$ is FAb. 
\end{thm}

One familiar family of $R$-analytic groups is the class of Chevalley linear algebraic groups over $R$. 
Here we have a stronger conclusion than that available in the general setting of Theorem \ref{Rstandarddiam}: 
the degree $C_2$ in the diameter bound may be taken to be \emph{independent of the dimension}. 

\begin{thm} \label{chevdiam}
Let $(R,\mathcal{M})$ be a commutative unital discrete valuation pro-$p$ domain, 
with $\mathcal{M}$ generated by $\mathcal{P}$. 
Let $G \leq GL_d (R)$ be the adjoint Chevalley group of type 
$X_l \in \lbrace A_l , B_l , C_l , D_l , E_6 , E_7 , E_8 , F_4 , G_2 \rbrace$ over $R$. 
Suppose $(X_l,p) \notin \lbrace (A_1,2),(B_l,2),(C_l,2),(D_l,2) \rbrace$. 
Let $K_n = G \cap (I_d + \mathcal{P}^n \mathbb{M}_d (R))$. 
Then there exist $C_1 (G) > 0$ and an absolute constant $C_2 > 0$ such that:
\begin{center}
$diam(G/K_n) \leq C_1 (log \lvert G / K_n \rvert)^{C_2}$. 
\end{center}
Moreover, the same bound holds for $G=SL_d (R),SO_d (R) \text{ or }Sp_d (R)$ 
provided $p \geq 3$, and for $G=SL_d (R)$ with $p=2$ provided $d \geq 3$. 
\end{thm}

Recall the correspondence between the classical root system of type $X_l$ 
and the associated adjoint Chevalley group $G$: 
if $X_l=A_l$ then $G = PSL_d (R)$, and in particular if $X_l=A_1$ 
then $G = PSL_2 (R)$; if $X_l=B_l$ or $D_l$ then $G=PSO_d(R)$ 
(with the dichotomy between $B_l$ and $D_l$ corresponding to the parity of $d$); 
if $X_l=C_l$ then $G=PSp_d(R)$. 

In the case $R = \mathbb{Z}_p$, this result was proved by Dinai \cite{Dinai}, 
under the additional hypothesis $p > max \lbrace \frac{l+2}{2},19 \rbrace$. 

Finally we consider a class of non-linear examples. 
Recall that, for $R$ a commutative unital ring, the \emph{Nottingham group} $\mathcal{N}(R)$ 
of $R$ is the set of formal power series over $R$ with constant coefficient $0$ and 1st order coefficient 1, 
with the operation of formal composition of power series. That is, an element $f \in \mathcal{N}(R)$ has the form:
\begin{center}
$f(t) = t + \sum_{k=2} ^{\infty} \lambda_k t^k$
\end{center}
for some $\lambda_k \in R$, and for $g \in \mathcal{N}$, 
\begin{center}
$f \cdot g = g(t + \sum_{k=2} ^{\infty} \lambda_k t^k)$. 
\end{center}
We take $R = \mathbb{F}_q$ a finite field (for $q$ a power of the prime $p$) 
and write $\mathcal{N}_q$ for $\mathcal{N}(\mathbb{F}_q)$. 
$\mathcal{N}_q$ is often used as a test case for more general techniques or conjectures in pro-$p$ group theory. 
The reason for this is twofold: first, computations in $\mathcal{N}_q$ 
can be made reasonably explicitly and simply. 
Second, $\mathcal{N}_q$ has extreme properties among pro-$p$ groups: 
as was proved first by Camina \cite{Camina1} and then by Fesenko \cite{Fesenko}, 
every countably based pro-$p$ group embeds as a closed subgroup of $\mathcal{N}_q$. 
We shall be concerned with the filtration by open normal subgroups:
\begin{center}
$K_n = \lbrace t + \sum_{k=n+1} ^{\infty} \lambda_k t^k \in \mathcal{N}_q \rbrace$
\end{center}
(so that in particular $K_1 = \mathcal{N}_q$). 

\begin{thm} \label{nottinghamdiam}
Suppose $p \geq 3$. 
Then there exist $C_1 (q)>0$ and an absolute constant $C_2 > 0$ such that:
\begin{center}
$diam(\mathcal{N}_q /K_n) \leq C_1 (log \lvert G / K_n \rvert)^{C_2}$. 
\end{center}
\end{thm}

As an application of these results, 
we make some elementary observations about mixing times of random walks in the finite groups we study. 
The direct relationship between diameter and spectral gap was recently exploited by Varj\'{u}, 
who in \cite{Varju} used a representation-theoretic argument to produce uniform weak spectral gap estimates for 
$SL_d (\mathbb{Z}/p^n \mathbb{Z})$, and deduced polylogarithmic diameter bounds. 
Here we reverse the direction of the argument, and deduce weak spectral gap estimates from uniform diameter bounds. 
As an aside, the diameter estimates for $SL_d (\mathbb{Z}/p^n \mathbb{Z})$ 
obtained by Varj\'{u} are uniform in $p$ but not in $d$, 
whereas those obtained via the Solovay-Kitaev procedure are uniform in $d$ but not in $p$. 
It may be instructive to apply Varj\'{u}'s method to other of the finite groups we treat in this paper 
to obtain diameter or spectral gap estimates which are similarly complementary to those given here. 

Let $S \subseteq \Gamma$ be a finite symmetric set. 
Let $X_1 , X_2 , \ldots$ be a sequence of independent random variables, each with law:
\begin{center}
$\frac{1}{\lvert S \rvert} \chi_S \in l^2 (\Gamma)$. 
\end{center}
For $l \in \mathbb{N}$, the \emph{simple random walk on $(\Gamma,S)$ at time $l$} 
is the random variable $Y_l = X_1 \cdots X_l$. 

For $(R,\mathcal{M})$ a discrete valuation pro-$p$ domain; 
$G$ a $d$-dimensional $R$-standard group; $x_1 , \ldots , x_d$ an $R$-basis for $\mathcal{M}^{(d)}$ 
(a set of so-called \textquotedblleft co-ordinates of the first kind\textquotedblright)
and $S \subseteq G$ a finite symmetric subset, 
we may express the simple random walk on $(G,S)$ by: 
\begin{center}
$Y_l = L_1 ^{(l)} x_1 + \cdots + L_d ^{(l)} x_d$
\end{center}
for some random variables $L_1 ^{(l)}, \ldots , L_d ^{(l)}$ supported on $R$. 

\begin{coroll} \label{rwstandard}
Suppose $\mathcal{L}_{G}$ is perfect and $S \subseteq G$ generates a dense subgroup. 
Then there exists $C(d) > 0$, such that for any $C' > 0$ 
there exists $C'' (G, \lvert S \rvert,C') > 0$ and $C'''(d,\lvert R/\mathcal{M} \rvert, C')>0$ such that, 
for any \linebreak$(\lambda_1 , \ldots , \lambda_d) \in R^{(d)}$, and for any $N \in \mathbb{N}$, we have:
\begin{center}
$\Big\lvert \mathbb{P} \big[ \lVert L_1 ^{(l)} - \lambda_1 \rVert , \ldots , \lVert L_d ^{(l)} - \lambda_d \rVert \leq c^{N+1} \big] - \frac{1}{\lvert R / \mathcal{M} \rvert^{dN}} \Big\rvert 
\leq e^{-C''' N^{C'}}$
\end{center}
whenever $l \geq C'' N^{C+C'}$. 
\end{coroll}

In other words, for such $l$ the probability that $Y_l$ is close to any element of $\mathcal{M}^{(d)}$ 
is nearly constant, with error at most $e^{-C''' N^{C'}}$. 

For a $d$-dimensional uniform pro-$p$ group $G$, an alternative representation for elements is available: 
for any $g \in G$ and any minimal (ordered) generating set $a_1 , \ldots , a_d$ for $G$, 
there exist $\mu_1 , \ldots , \mu_d \in \mathbb{Z}_p$ such that 
$g = a_1 ^{\mu_1} \cdots a_d ^{\mu_d}$ 
(so-called \textquotedblleft co-ordinates of the second kind\textquotedblright). We therefore have:
\begin{center}
$Y_l = a_1 ^{M_1 ^{(l)}} \cdots a_d ^{M_d ^{(l)}}$
\end{center}
for some random variables $M_1 ^{(l)} , \ldots , M_d ^{(l)}$ supported on $\mathbb{Z}_p$. 

\begin{coroll} \label{rwpadic}
Let $p \geq 3$. Suppose $G$ is uniform and FAb and $S \subseteq G$ generates a dense subgroup. 
Then there exists $C(d) > 0$, such that for any $C' > 0$ there exists 
$C'' (G,\lvert S \rvert,C') > 0$ and $C'''(d,p,C')>0$ such that, 
for any $\mu_1 , \ldots , \mu_d \in \mathbb{Z}_p$, and for any $N \in \mathbb{N}$, we have:
\begin{center}
$\Big\lvert \mathbb{P} \big[ \lVert M_1 ^{(l)} - \mu_1 \rVert , \ldots , \lVert M_d ^{(l)} - \mu_d \rVert \leq p^{-N-1} \big] - \frac{1}{p^{dN}} \Big\rvert 
\leq e^{-C''' N^{C'}}$
\end{center}
whenever $l \geq C'' N^{C+C'}$. 
\end{coroll}

In the Nottingham group $\mathcal{N}_q$, the question of mixing times for the groups $\mathcal{N}_q / K_n$ 
was raised by Diaconis \cite{Diaconis}. We may express: 
\begin{center}
$Y_l = t + \sum_{i = 2} ^{\infty} A_i ^{(l)} t^i$
\end{center}
for some random variables $A_i ^{(l)}$ supported on $\mathbb{F}_q$. 

\begin{coroll} \label{rwnottingham}
Let $p \geq 3$. Suppose $S \subseteq \mathcal{N}_q$ generates a dense subgroup. 
Then there exists an absolute constant $C > 0$, such that for any $C' > 0$ 
there exists $C'' (q , \lvert S \rvert , C') > 0$ 
and $C''' (q , C')>0$ such that, 
for any sequence $(\alpha_i)_i$ in $\mathbb{F}_q$, 
and for any $N \in \mathbb{N}$, we have:
\begin{center}
$\Big\lvert \mathbb{P} [A_2 ^{(l)} = \alpha_2 , \ldots , A_N ^{(l)} = \alpha_N] - \frac{1}{q^{N-1}} \Big\rvert 
\leq e^{-C''' N^{C'}}$
\end{center}
whenever $l \geq C'' N^{C+C'}$. 
\end{coroll}

\subsection{Background}

The work of estimating the diameter and spectral gap for finite groups 
with respect to various generating sets has been going on for many years: 
see for instance \cite{Diaconis} for an overview of some of the work on \emph{card-shuffling} problems, 
that is, questions of mixing and diameter in the symmetric group $Sym(n)$. 

In the past decade however, there has been a flood of results which provide diameter or spectral gap 
estimates for finite simple groups of Lie type, 
and which systematically treat all (or at least most) generating sets simultaneously. 
In many ways this programme was begun by Helfgott \cite{Helf1} 
who established polylogarithmic diameter bounds for $G=PSL_2 (p)$, independent of $S$. 
These bounds were deduced from lower bounds on the growth of an arbitrary generating set 
under multiplication with itself. 
A series of papers by many authors quickly followed, 
many expressed in the language of \emph{approximate groups}, 
which generalised Helfgott's work to arbitrary finite simple groups of Lie type 
(see \cite{BGT}, \cite{PySz} for the most general statements). 

A key motivation for the development of this field was the discovery, 
first made by Bourgain and Gamburd in \cite{BoGa0}, 
that such growth results could be harnessed to construct new examples of \emph{expanders}. 
These are sequences of pairs $(G_n,S_n)_n$ 
for which the spectral gap is bounded below, independent of $n$. 
In particular, for such sequences $diam(G_n,S_n)$ is logarithmic in $\lvert G_n \rvert$. 
Bourgain and Gamburd deduced from Helfgott's result that $(PSL_2 (p))_p$ 
is an expander with respect to \emph{random} generators. 
With the proliferation of growth results and the popularization of the Bourgain-Gamburd philosophy 
came a corresponding set of papers producing new examples of expanders, 
culminating in the recent work of Breuillard, Green, Guralnick and Tao \cite{BGGT}, 
who showed that \emph{any} sequence of finite simple groups of Lie type of bounded rank 
is an expander with respect to random generators. 

For finite groups $G$ which arise as images of linear groups over pro-$p$ rings, 
the situation is very different from in the simple case: a group such as 
$SL_d (\mathbb{Z} / p^n \mathbb{Z})$ has many large normal subgroups, 
arising as the kernels of congruence maps $SL_d (\mathbb{Z} / p^n \mathbb{Z}) \rightarrow SL_d (\mathbb{Z} / p^m \mathbb{Z})$ for $m \leq n$. The presence of such subgroups is both a blessing and a curse. 
On the one hand, a clean statement about the growth of arbitrary subsets \`{a} la Helfgott becomes less accessible 
(there are in some sense too many subgroups in which a generating set may become partially trapped). 
On the other, the filtration by the congruence kernels 
opens the way to arguments by induction on the level of the filtration. 
One such is the \emph{Solovay-Kitaev Procedure}, 
originally applied to $SU(d)$ in the study of compilers in quantum computation \cite{DaNi}, 
but equally valid in the profinite world. 
This procedure works by exploiting the commutator structure of the groups concerned: 
approximating elements at lower levels in the filtration 
by commutators of elements at higher levels. 

Several papers have already exploited this idea: 
Gamburd and Shahshahani \cite{GaSh} used it to establish upper bounds on $diam(SL_2 (\mathbb{Z} / p^n \mathbb{Z}))$. 
Their analysis was extended by Dinai \cite{Dinai} to arbitrary Chevalley groups over $\mathbb{Z} / p^n \mathbb{Z}$, 
with bounds independent of the rank of the Chevalley group scheme. 
Finally Bourgain and Gamburd (\cite{BoGa1} and \cite{BoGa2}) 
combined a Solovay-Kitaev-type argument with results on random matrix products and the sum-product phenomenon in the ring $\mathbb{Z} / p^n \mathbb{Z}$ to produce many new examples of expander Cayley graphs of 
$SL_d (\mathbb{Z} / p^n \mathbb{Z})$ (though without uniformity in $d$). 
In fact, the ideas explored in \cite{GaSh} and \cite{Dinai} are relevant to a much broader class of groups. 
It is the goal of this paper to present the Solovay-Kitaev procedure for profinite groups in an appropriate level of generality and to exploit it for uniform diameter bounds in families of finite groups which have not been considered previously. 

The paper is structured as follows. In Section \ref{SKSection} 
we discuss analogues of the Solovay-Kitaev Procedure for profinite groups upon which all our results will be based. 
In Section \ref{classicalsection} we prove Theorem \ref{chevdiam} in the case of classical groups. 
This is achieved via a very concrete analysis of the Lie algebras of these groups, 
in their standard matrix representation, and does not require any understanding of the associated root systems. 
In Section \ref{AnalyticSection} we study the Lie algebras of $R$-analytic groups, 
prove Theorem \ref{Rstandarddiam} and deduce both Theorem \ref{padicdiam} and the exceptional case of Theorem \ref{chevdiam}. In Section \ref{NottinghamSection} we prove Theorem \ref{nottinghamdiam}. 
Consequences of these results for mixing times of random walks are explained in Section \ref{RWSection}. 

I am deeply grateful to my supervisor, Marc Lackenby, for suggesting that I investigate diameters of $p$-groups; 
for his many suggestions concerning this project and for his continued support and enthusiasm for my research. 
I am also grateful to EPSRC for providing financial support during the undertaking of this work. 
Several results from this paper were first presented at the Postgraduate Conference in Group Theory, 
held at the University of Birmingham in June 2014. 
I would like to thank the organizers of that conference for their hard work and for providing me with a warm welcome. 

\section{The Profinite Solovay-Kitaev Procedure} \label{SKSection}

In this section we prove general results about the diameters of finite quotients of a finitely generated profinite group $G$ under some hypotheses on the behaviour of commutators in $G$. 
The proofs of Theorems \ref{Rstandarddiam}, \ref{chevdiam} and \ref{nottinghamdiam} will thereby be reduced to a verification that commutators in the groups concerned satisfy these hypotheses. 
Our first result in this direction, which will also serve as a warm-up for the more general technical result required for some applications, is: 

\begin{propn} \label{SolovayKitaev1}
Let $G$ be a profinite group, 
$(K_n)_{n \geq 1}$ a descending sequence of open normal subgroups of $G$. 
Suppose: 
\begin{itemize}
\item[(i)] For all $m,n \geq 1$, $[K_m , K_n] \subseteq K_{m+n}$; 
\item[(ii)] There exists $n_0 \geq 1$ such that for all $m,n \geq n_0$ satisfying $n \leq m \leq 2n$, 
and all $g \in K_{n+m}$, there exist: 
\begin{center}
$g_1 , \ldots , g_A \in K_n$ , $h_1 , \ldots , h_A \in K_m$
\end{center}
such that $[g_1 , h_1] \cdots [g_A , h_A] g^{-1} \in K_{2n+m}$. 
\end{itemize}
Then $G / \bigcap_{n=1} ^{\infty} K_n$ is finitely generated and there exists $C > 0$ 
(depending only on $A$, $\lvert G / K_{2 n_0} \rvert$) such that for all $n \geq 1$, 
\begin{center}
$diam(G/K_n) \leq C n^{\frac{log(8A^2 +6A)}{log(2)}}$. 
\end{center}
\end{propn}

\begin{rmrk}
In all the examples we consider below, 
we will have in addition that the sequence $(\lvert K_i / K_{i+1} \rvert)_i$ is constant, 
so that a bound for $diam(G/K_n)$, which is polynomial in $n$, 
is polylogarithmic in $\lvert G/K_n \rvert$. 
\end{rmrk}

If we imagine the subgroups $K_n$ to be balls in $G$ around the identity of radius $c^n$, 
for some $c \in (0,1)$, then hypothesis (i) of Proposition \ref{SolovayKitaev1} 
says roughly that a commutator of two elements is of size quadratic in the sizes of those two elements, 
whereas hypothesis (ii) says that every element may be approximated by a product of (a bounded number of) 
commutators of larger elements. Indeed, in a real Lie group such as $SU_d$, 
replacing the $K_n$ with Euclidean balls around $Id$ and interpreting \textquotedblleft size\textquotedblright as \textquotedblleft Euclidean distance\textquotedblright, 
we recover the properties on which the proof of the original Solovay-Kitaev Theorem is based. 
In this sense then, it is legitimate to describe Proposition \ref{SolovayKitaev1} 
as a \textquotedblleft profinite Solovay-Kitaev Theorem\textquotedblright. 

In fact, rather than hypothesis (i) itself the proof uses a reformulation (i'), 
as explained in the following Lemma. Under the interpretation just outlined, 
hypothesis (i') says that, given a pair of elements $g , h$ 
and a pair of \textquotedblleft approximations\textquotedblright $g' , h'$ up to some error, 
$[g',h']$ approximates $[g,h]$ up to an error which is quadratic in the sizes of $g$ and $h$, 
and the errors in the original approximations $g'$ and $h'$. 

\begin{lem} \label{SKlem1}
Let $G$ be a profinite group, 
$(K_n)_{n \geq 1}$ a descending sequence of open normal subgroups. 
The following conditions are equivalent:
\begin{itemize}
\item[(i)] For all $m,n \geq 1$, $[K_m , K_n] \subseteq K_{m+n}$. 
\item[(i')] For all $m,m',n,n' \geq 1$, with $m \leq m'$, $n \leq n'$, 
and for all $g,g' \in K_n$; $h,h' \in K_m$ with $g^{-1}g' \in K_{n'}$; $h^{-1}h' \in K_{m'}$, 
\begin{center}
$[g,h]^{-1}[g',h'] \in K_{min(m+n',m'+n)}$. 
\end{center}
\end{itemize}
\end{lem}

\begin{proof}
Assuming (i), write $\tilde{g}=g^{-1}g',\tilde{h}=h^{-1}h'$. Then we may express $[g',h']$ as:
\begin{center}
$[g',h']=[g,\tilde{h}][g,h][[g,h],\tilde{h}][[g,h \tilde{h}],\tilde{g}][\tilde{g},h \tilde{h}]$
\end{center}
by standard commutator identities. Now:
\begin{center}
$[g,\tilde{h}] \in K_{n+m'}$; $[[g,h],\tilde{h}] \in K_{n+m+m'}$; 
$[[g,h \tilde{h}],\tilde{g}] \in K_{n+n'+m}$; $[\tilde{g},h \tilde{h}] \in K_{n'+m}$
\end{center}
by (i), so that $[g,h] \equiv [g',h'] \text{ mod } K_{min(m+n',m'+n)}$. 
\\ \\
Conversely, assuming (i'), let $g \in K_n$, $h \in K_m$. 
We may assume $n \leq m$.\linebreak Then $g^{-1} h \in K_n$. 
Taking $n'=n$, $m' > m$ in (i'), we have \linebreak $min(m+n',m'+n)=n+m$, 
so we may set $g' = h' = h$ to obtain: 
\begin{center}
$e = [h,h] \equiv [g,h] \text{ mod } K_{n+m}$. 
\end{center}
In other words, $[g,h] \in  K_{n+m}$, as required. 
\end{proof}

The diameter bound will come from the following Lemma, 
the conditions of which we shall verify in the setting of Proposition \ref{SolovayKitaev1}. 

\begin{lem} \label{SKlem2}
Let $G$ be a profinite group, 
$(K_n)_{n \geq 1}$ a descending sequence of open normal subgroups. 
Suppose there exist $n_0 , B , D \in \mathbb{N}$, with $D \geq 2$, such that, 
for every $n \geq n_0$ and every $X \subseteq G$, 
\begin{equation} \label{telescopingquotients}
K_n / K_{Dn} \subseteq K_{Dn} X / K_{Dn} \Rightarrow K_{Dn} / K_{D^2 n} \subseteq K_{D^2 n} X^B / K_{D^2 n}. 
\end{equation}
Then $G / \bigcap_{n=1} ^{\infty} K_n$ is finitely generated and there exists $C > 0$ 
(depending only on $B , \lvert G / K_{Dn_0} \rvert$) such that for any $n \in \mathbb{N}$, 
\begin{center}
$diam(G/K_n) \leq C n^{\frac{log(B)}{log(D)}}$. 
\end{center}
\end{lem}

\begin{proof}
Let $S \subseteq G$, and suppose the restriction of the natural epimorphism $\pi_{Dn_0} : G \twoheadrightarrow G/K_{Dn_0}$ to $\langle S \rangle$ is surjective. Then for some $l_0 \in \mathbb{N}$ (independent of $S$), 
\begin{center}
$K_{Dn_0} B_S (l_0 ) / K_{Dn_0} = G / K_{Dn_0}$
\end{center}
(we may always take $l_0 \leq \lvert G / K_{Dn_0} \rvert$). 
In particular we have\linebreak $K_{n_0} / K_{Dn_0} \subseteq K_{Dn_0} B_S (l_0 ) / K_{Dn_0}$. 
By an easy induction involving (\ref{telescopingquotients}), we have for any $i \in \mathbb{N}$, 
\begin{center}
$K_{D^i n_0} / K_{D^{i+1} n_0} \subseteq K_{D^{i+1} n_0} B_S (B^i l_0 ) / K_{D^{i+1} n_0}$. 
\end{center}
It follows that, for any $n \leq D^i n_0$, 
\begin{center}
$diam(G/K_n , S) \ll_{B,l_0} B^i$. 
\end{center}
Hence for arbitrary $n$, choosing $i$ such that $D^{i-1} n_0 \leq n \leq D^i n_0$, 
\begin{center}
$diam(G/K_n , S) \ll_{B,l_0} B^{\frac{log (n)}{log(D)}} = n^{\frac{log(B)}{log(D)}}$. 
\end{center}
Now let $\overline{S} \subseteq G / K_n$ and suppose $\langle \overline{S} \rangle = G / K_n$. 
If $n \leq Dn_0$,\linebreak then $diam(G/K_n , \overline{S}) \leq l_0$. 
Otherwise, the image of $\overline{S}$ in $G / K_{Dn_0}$ is a generating set, and the preceding argument applies. 
\\ \\
In particular, let $\tilde{S} \subseteq G$ be finite with image in $G / K_{Dn_0}$ a generating set. 
Then for every $n$, $\tilde{S}$ generates $G$ modulo $K_n$, 
so $\tilde{S}$ maps to a topological generating set in $G / \bigcap_{n=1} ^{\infty} K_n$. 
\end{proof}

\begin{proof}[Proof of Proposition \ref{SolovayKitaev1}]
Let $n \geq n_0$. Suppose $X \subseteq G$ is such that: 
\begin{equation} \label{approx1}
K_n / K_{2n} \subseteq K_{2n} X / K_{2n}. 
\end{equation}
Let $g \in K_{2n}$. By hypothesis (ii) there exist $g_1 , \ldots , g_A, h_1 , \ldots , h_A \in K_n$ such that:
\begin{center}
$g \equiv [g_1 , h_1] \cdots [g_A , h_A] \text{ mod } K_{3n}$. 
\end{center}
By (\ref{approx1}) there exist $g_1 ', \ldots , g_A ', h_1 ', \ldots , h_A ' \in X$ with 
$g_i \equiv g_i ' , h_i \equiv h_i ' \text{ mod } K_{2n}$ for $i=1 , \ldots , A$. 
By hypothesis (i') from Lemma \ref{SKlem1}, $[g_i , h_i] \equiv [g_i ', h_i '] \text{ mod } K_{3n}$. 
Hence $g \equiv [g_1' , h_1'] \cdots [g_A' , h_A'] \text{ mod } K_{3n}$, so that:
\begin{equation} \label{approx2}
K_{2n} / K_{3n} \subseteq K_{3n} X^{4A} / K_{3n}. 
\end{equation}
Likewise, let $g \in K_{3n}$. There exist $g_1 , \ldots , g_A \in K_{n}$, $h_1 , \ldots , h_A \in K_{2n}$ such that:
\begin{center}
$g \equiv [g_1 , h_1] \cdots [g_A , h_A] \text{ mod } K_{4n}$. 
\end{center}
By (\ref{approx1}) and (\ref{approx2}) there exist 
$g_1', \ldots , g_A' \in X$ and $h_1', \ldots , h_A' \in X^{4A}$ such that 
$g_i \equiv g_i' \text{ mod } K_{2n}$ and $h_i \equiv h_i' \text{ mod } K_{3n}$, 
so that $[g_i,h_i] \equiv [g_i',h_i'] \text{ mod } K_{4n}$ for $i = 1 , \dots , A$ and: 
\begin{center}
$g \equiv [g_1' , h_1'] \cdots [g_A' , h_A'] \text{ mod } K_{4n}$. 
\end{center}
Hence: 
\begin{equation} \label{approx3}
K_{3n} / K_{4n} \subseteq K_{4n} X^{8A^2 + 2A} / K_{4n}. 
\end{equation}
Combining (\ref{approx1}), (\ref{approx2}) and (\ref{approx3}), we obtain 
$K_{2n} / K_{4n} \subseteq K_{4n} X^{8A^2 +6A} / K_{4n}$. 
The required result now follows from Lemma \ref{SKlem2}, applied with $B = 8A^2 +6A$, $D = 2$. 
\end{proof}

Proposition \ref{SolovayKitaev1} will suffice to prove Theorem \ref{chevdiam} 
in the case of classical groups over pro-$p$ rings. 
For general analytic pro-$p$ groups and for the Nottingham group, however, 
generating elements as products of commutators is more difficult. 
For example, $[K_n,K_m]$ may not be the whole of $K_{n+m}$ 
(as will always be the case in the setting of Proposition \ref{SolovayKitaev1}) 
but some deeper subgroup $K_{n+m+k}$ (with $k \geq 1$ bounded independent of $m,n$). 
To circumvent these and other complexities of the general case, 
we prove a stronger version of Proposition \ref{SolovayKitaev1}, 
in which hypothesis (ii) has been weakened: 

\begin{propn} \label{SolovayKitaev2}
Let $G$ be a profinite group, 
$(K_n)_{n \geq 1}$ a descending sequence of open normal subgroups of $G$. 
Suppose: 
\begin{itemize}
\item[(i)] For all $m,n \geq 1$, $[K_m , K_n] \subseteq K_{m+n}$; 
\item[(ii)] There exists $\epsilon \in (0,1)$; $A,M_1 ,M_2 \in \mathbb{N}$ such that for all $n \geq M_1$, 
there exist $n_i , m_i \in \mathbb{N}$ (for $i=1,2,3$) with:
\begin{center}
$\frac{n}{3} (2+i+ \epsilon) \leq n_i \leq m_i \leq \frac{2n}{3} (2+i)$; $n_i + m_i = (2+i)n-M_2$
\end{center}
and for all $g \in K_{(2+i)n}$, there exist: 
\begin{center}
$g_1 , \ldots , g_A \in K_{n_i}$ , $h_1 , \ldots , h_A \in K_{m_i}$
\end{center}
such that $[g_1 , h_1] \cdots [g_A , h_A] g^{-1} \in K_{(2+i)n + n_i - M_2}=K_{2n_i + m_i}$. 
\end{itemize}
Then $G / \bigcap_{n=1} ^{\infty} K_n$ is finitely generated and there exists $C > 0$ 
(depending on $A$, $\lvert G / K_{3n_0} \rvert$, 
where $n_0 = max \lbrace 2 M_1 ,\lceil \frac{3M_2}{\epsilon} \rceil \rbrace$) such that: 
\begin{center}
$diam(G/K_n) \leq C n^{\frac{6 log(4A+1)}{log(3)}}$. 
\end{center}
\end{propn}

\begin{proof}
First claim that for any $n \geq max \lbrace 2 M_1 ,\frac{3M_2}{\epsilon} \rbrace$ and any $X \subseteq G$, 
\begin{equation} \label{clunkyclaim}
K_n / K_{3n} \subseteq K_{3n} X / K_{3n} \Rightarrow K_n / K_{6n} \subseteq K_{6n} X^{(4A+1)^3} / K_{6n}. 
\end{equation}
Let $g \in K_{3n}$. By hypothesis $(ii)$, 
there exist $g_1 , \ldots , g_A \in K_{n_1}$,\linebreak $h_1 , \ldots , h_A \in K_{m_1}$ such that:
\begin{center}
$g \equiv [g_1,h_1] \cdots [g_A,h_A] \text{ mod } K_{3n + n_1 - M_2}$. 
\end{center}
By assumption, there exist $g_1' , \ldots , g_A' , h_1' , \ldots , h_A' \in X$ 
such that $g_i \equiv g_i'$, $h_i \equiv h_i' \text{ mod } K_{3n}$, 
so that $g_i ' \in K_{n_1}$, $h_i ' \in K_{m_1}$. 
By Lemma \ref{SKlem1}, 
\begin{center}
$[g_i,h_i] \equiv [g_i',h_i'] \text{ mod } K_{3n + n_1}$. 
\end{center}
Hence $g \equiv [g_1',h_1'] \cdots [g_A',h_A'] \text{ mod } K_{3n + n_1 - M_2}$. Therefore:
\begin{center}
$K_{3n} / K_{3n+n_1 - M_2} \subseteq K_{3n+n_1 - M_2} X^{4A} / K_{3n+n_1 - M_2}$
\end{center}
and, combining with the hypothesis $K_n / K_{3n} \subseteq K_{3n} X / K_{3n}$, 
\begin{center}
$K_n / K_{3n+n_1 - M_2} \subseteq K_{3n+n_1 - M_2} X^{4A+1} / K_{3n+n_1 - M_2}$. 
\end{center}
In particular, since $n_1 \geq n + \frac{\epsilon n}{3} \geq n + M_2$, $K_n / K_{4n} \subseteq K_{4n} X^{4A+1} / K_{4n}$. 
\\ \\
We now simply repeat the same procedure: let $n_2 , m_2 \in \mathbb{N}$ be as above. 
We deduce: 
\begin{center}
$K_{4n} / K_{4n+n_2-M_2} \subseteq K_{4n+n_2-M_2} X^{4A(4A+1)} / K_{4n+n_2-M_2}$. 
\end{center}
Combining this estimate with $K_n / K_{4n} \subseteq K_{4n} X^{4A+1} / K_{4n}$, 
and since \linebreak$4n+n_2-M_2 \geq 5n$, we have:
\begin{center}
$K_n / K_{5n} \subseteq K_{5n} X^{(4A+1)^2} / K_{5n}$. 
\end{center}
Finally let $n_3 , m_3 \in \mathbb{N}$ be as above. We have: 
\begin{center}
$K_{5n} / K_{5n+n_3-M_2} \subseteq K_{5n+n_3-M_2} X^{4A(4A+1)^2} / K_{5n+n_3-M_2}$. 
\end{center}
Combining with $K_n / K_{5n} \subseteq K_{5n} X^{(4A+1)^2} / K_{5n}$, 
since $5n+n_3 \geq 6n$, the claim follows. 
\\ \\
In particular, (\ref{clunkyclaim}) implies that for $n \geq max \lbrace 2 M_1 ,\frac{3M_2}{\epsilon} \rbrace$, 
\begin{center}
$K_n / K_{3n} \subseteq K_{3n} X / K_{3n} \implies K_{2n} / K_{6n} \subseteq K_{6n} X^{(4A+1)^3} / K_{6n}$. 
\end{center}
Applying (\ref{clunkyclaim}) again, with $n$ replaced by $2n$ and $X$ replaced by $X^{(4A+1)^3}$, 
\begin{center}
$K_{2n} / K_{12n} \subseteq K_{12n} X^{(4A+1)^6} / K_{12n}$
\end{center}
so that in particular, $K_{3n} / K_{9n} \subseteq K_{9n} X^{(4A+1)^6} / K_{9n}$. 
The result now follows from Lemma \ref{SKlem2}, applied with $B = (4A+1)^6$, $D=3$. 
\end{proof}

The proof of Proposition \ref{SolovayKitaev2} is sufficiently robust that qualitatively similar 
(though quantitatively worse) diameter bounds should be available under even weaker hypotheses. 
We shall not pursue such results here, 
as the level of generality already achieved is sufficient for all the examples we shall consider. 
We conclude this section by noting some cases in which hypothesis (i) 
of Propositions \ref{SolovayKitaev1} and \ref{SolovayKitaev2} is always satisfied. 

\begin{ex} \label{Rings}
\begin{itemize}
\item[(i)] Let $G$ be \emph{any} pro-$p$ group; $K_n$ be the $n$th term of the lower central $p$-series for $G$. 

\item[(ii)] Let $R$ be a unital profinite ring; $G \leq R^*$; 
$I \vartriangleleft R$ a proper two-sided open ideal. 
Define $K_n = G \cap (1 + I^n) \vartriangleleft G$. Let $n,m \in \mathbb{N}$ 
with $n \leq m$ and let $g \in K_n$, $h \in K_m$. 
Let $a , \tilde{a} \in I^m$, $b , \tilde{b} \in I^n$ be such that: 
\begin{center}
$g=1+a$, $g^{-1}=1+\tilde{a}$, $h=1+b$, $h^{-1}=1+\tilde{b}$. 
\end{center}
Then $a+\tilde{a}+ \tilde{a} a = b+\tilde{b}+ \tilde{b} b = 0$, so: 
\begin{center}
$[g,h] \equiv 1+ab+\tilde{a}b+\tilde{a} \tilde{b}+ \tilde{b}a$ 

$\equiv 1 + ab-ba \text{ mod } I^{2n + m}$. 
\end{center}
In particular, $[g,h] \in K_{n+m}$. 

\item[(iii)] As a particular case of (ii), letting $R = \mathbb{F}_p G$ and $I \vartriangleleft R$ 
be the augmentation ideal, $K_n$ is the $n$th mod-$p$ dimension subgroup of $G$. 
\end{itemize}
\end{ex}

\section{Classical Groups over $R$} \label{classicalsection}

In this section we prove Theorem \ref{chevdiam} in the case for which $X_l$ is classical, 
so that the associated adjoint Chevalley group over $R$ is one of 
$PSL_d (R)$, $PSO_d (R)$, or $PSp_{d} (R)$ (with $d$ even in the latter case). 
To be more precise,\linebreak we prove the diameter bound for $G = SL_d (R) , SO_d (R)$ 
or $Sp_{d} (R)$; \linebreak$K_n = G \cap (I_d + \mathcal{P}^n \mathbb{M}_d (R))$. 
It shall be useful at this point to make a general observation, 
to the effect that $diam$ behaves well with respect to extensions. 

\begin{lem} \label{easyextnlem}
Let $G$ be a finite group, $K \vartriangleleft G$. Then:
\begin{itemize}
\item[(i)] $diam(G/K) \leq diam(G)$. 
\item[(ii)] $diam(G) \leq (2 \cdot diam(G/K)+1) \cdot (diam(K)+\frac{1}{2})-\frac{1}{2}$. 
\end{itemize}
\end{lem}

\begin{proof}
\begin{itemize}
\item[(i)] is straightforward. 
\item[(ii)] Let $S \subseteq G$ be a generating set. 
Then $B_S (diam(G/K))$ contains a transversal $T$ to $K$ in $G$, with $1 \in T$. 
By the Reidemeister-Schreier process, $B_S (2 \cdot diam(G/K) + 1)$ 
contains a generating set for $K$. Hence: 
\begin{center}
$diam(G,S) \leq diam(G/K)+diam(K) \cdot (2 \cdot diam(G/K)+1)$
\end{center}
as required. 
\end{itemize}
\end{proof}

The required result for the adjoint form then follows straightforwardly: 
letting $\rho : G \rightarrow GL_D (R)$ be the adjoint representation of $G$ 
on the associated Lie algebra (of dimension $D$), for any $g \in G$, 
if $g \equiv I_d \text{ mod } \mathcal{P}^n$ then \linebreak$\rho(g) \equiv I_D \text{ mod } \mathcal{P}^n$. 
Thus letting $K_n = G \cap (I_d + \mathcal{P}^n \mathbb{M}_d (R))$, \linebreak
$L_n = \rho(G) \cap (I_D + \mathcal{P}^n \mathbb{M}_D (R))$, 
$\rho$ descends to an epimorphism \linebreak$G / K_n \twoheadrightarrow \rho(G) / L_n$. 
By Lemma \ref{easyextnlem} (i), 
\begin{center}
$diam (\rho(G) / L_n) \leq diam(G / K_n) \leq C_1 (log \lvert G / K_n \rvert)^{C_2}$. 
\end{center}
The polylogarithmic diameter bound in $\lvert G / K_n \rvert$ then translates 
to a polylogarithmic bound in $\lvert \rho(G) / L_n \rvert$ (with possibly larger constant $C_1$). 
For $\lvert G / K_n \rvert \ll \lvert R/ \mathcal{M} \rvert^{d^2 n}$ 
and $\lvert \rho(G) / L_n \rvert \gg \lvert R/ \mathcal{M} \rvert^n$. 
\\ \\
We verify the hypotheses of Proposition \ref{SolovayKitaev1} 
for $G = SL_d (R)$, $SO_d (R)$, or $Sp_{d} (R)$. 
Recall that we permit ourselves the assumption that $p \geq 3$ unless $G = SL_d (R)$ and $d \geq 3$. 
Hypothesis (i) follows immediately from Example \ref{Rings} (ii). 
Moreover, for $g \in K_n$, $h \in K_m$, with $n \leq m \leq 2n$, writing:
\begin{center}
$g = I_d + \mathcal{P}^n X$; $h = I_d + \mathcal{P}^m Y$
\end{center}
for some $X,Y \in \mathbb{M}_d (R)$, we have: 
\begin{center}
$[g,h] \equiv I_d + \mathcal{P}^{m+n}(X,Y) \text{ mod } \mathcal{P}^{m+2n}$
\end{center}
where $(X,Y)=XY-YX$ is the Lie bracket. Hence for $g_1 , \ldots , g_A \in K_n$, $h_1 , \ldots , h_A \in K_m$, 
writing $g_i=I_d + \mathcal{P}^n X_i$, $h_i=I_d + \mathcal{P}^m Y_i$, we have: 
\begin{center}
$[g_1,h_1] \cdots [g_A,h_A] 
\equiv I_d + \mathcal{P}^{m+n}((X_1,Y_1)+ \ldots +(X_A,Y_A)) \text{ mod }\mathcal{P}^{m+2n}$. 
\end{center}
To verify hypothesis (ii) of Proposition \ref{SolovayKitaev1}, 
it therefore suffices to find $A \in \mathbb{N}$ (independent of $G$) such that, for any $g \in K_{m+n}$, 
we can find $X_1 , \ldots X_A$, $Y_1 , \ldots , Y_A \in \mathbb{M}_d (R)$ such that:
\begin{itemize}
\item[(a)] $g-I_d \equiv \mathcal{P}^{m+n}((X_1,Y_1)+ \ldots +(X_A,Y_A)) \text{ mod }\mathcal{P}^{m+2n}$; 
\item[(b)] There exist $g_1 , \ldots , g_A \in K_n$, $h_1 , \ldots , h_A \in K_m$ such that
\begin{center}
$g_i-I_d \equiv \mathcal{P}^n X_i \text{ mod } \mathcal{P}^{2n}$, 
$h_i-I_d \equiv \mathcal{P}^m Y_i \text{ mod } \mathcal{P}^{2m}$ 
\end{center}
for $1 \leq i \leq A$. 
\end{itemize}
As in the statement of Proposition \ref{SolovayKitaev1}, 
finding $A$ independent of $G$ yields an exponent $C_2$ in Theorem \ref{chevdiam} independent of $X_l$. 
For $G = SL_d (R)$, $SO_d (R)$ or $Sp_{d} (R)$, 
let $\mathfrak{g} = \mathfrak{sl}_d (R)$, $\mathfrak{so}_d (R)$ 
or $\mathfrak{sp}_d (R)$ be the associated Lie ring over $R$. 
Conditions (a), (b) above will follow straightforwardly from the following, 
which we verify for each group scheme in turn: 
\begin{itemize}
\item[(a')] For every $n \in \mathbb{N}$ and every $g \in K_n$, 
there exists $X \in \mathfrak{g}$ such that such that $g - I_d \equiv \mathcal{P}^n X \text{ mod } \mathcal{P}^{2n}$. 
\item[(b')] There exists $A \in \mathbb{N}$ (independent of $\mathfrak{g}$) 
such that every element of $\mathfrak{g}$ is the sum of at most $A$ brackets in $\mathfrak{g}$ 
(as we shall see, it suffices to take $A=3$). 
\item[(c')] There exists $\mathcal{B} \subseteq \mathfrak{g}$, generating $\mathfrak{g}$ as a $\mathbb{Z}$-module, 
such that for every $n \in \mathbb{N}$ and every $X \in \mathcal{B}$, 
there exists $g \in K_n$ such that such that $g - I_d \equiv \mathcal{P}^n X \text{ mod } \mathcal{P}^{2n}$. 
\end{itemize}
For, given $g \in K_{n+m}$, we immediately produce $X_i , Y_i$ as in (a) by applying (a'), (b') to $g$. 
Now writing an arbitrary element $Z \in \mathfrak{g}$ as $\sum_{i=1} ^r Z_i$, for $Z_i \in \mathcal{B}$, 
and letting $k_1 , \ldots , k_r \in K_l$ be such that 
$k_i - I_d \equiv \mathcal{P}^l Z_i \text{ mod } \mathcal{P}^{2l}$ as in (c'), we have: 
\begin{center}
$I_d + \mathcal{P}^l Z \equiv k_1 \cdots k_r \in K_l \text{ mod } \mathcal{P}^{2l}$. 
\end{center}
Applying this observation to $X_i , Y_i$ with $l = n,m$ respectively, we obtain $g_i , h_i$ as in (b). 

\subsection{$SL_d$}

Let $\mathfrak{sl}_d (R)$ denote the space of traceless $d \times d$ matrices over $R$; 
it is spanned over $R$ by the matrices $E_{i,j}$, $D_{a,b}$, 
for $i \neq j$, $a <b$, where:
\begin{center}
$(E_{i,j})_{r,s} = \delta_{i,r} \delta_{j,s}$, 
$(D_{a,b})_{r,s} = \delta_{a,r} \delta_{a,s} - \delta_{b,r} \delta_{b,s}$. 
\end{center}

\begin{itemize}
\item[(a')] Let $g \in K_n$. Write $g=I_d+\mathcal{P}^n X$, for some $X \in \mathbb{M}_d (R)$. Then:
\begin{center}
$1 = det(g) \equiv 1+\mathcal{P}^n tr(X) \text{ mod } \mathcal{P}^{2n}$
\end{center}
so $tr(X) \equiv 0 \text{ mod }\mathcal{P}^n$. Hence there exists 
$X' \in \mathfrak{sl}_d (R)$ such that \linebreak$X \equiv X' \text{ mod }\mathcal{P}^n$. 

\item[(b')] First suppose $d \geq 3$. 
Define the $R$-module endomorphisms \linebreak$T_1 , T_2 : \mathfrak{sl}_d (R) \rightarrow \mathfrak{sl}_d (R)$ by:
\begin{center}
$T_1 (X) = (X,\sum_{i=1} ^{d-1} E_{i+1,i})$, $T_2 (X) = (X,\sum_{i=1} ^{d-1} E_{i,i+1})$. 
\end{center}
Then: 
\begin{center}
$D_{j,j+1}=T_1 (E_{j,j+1})$ for $j=1, \ldots ,d-1$, 
\\
$E_{i,j-1}-E_{i+1,j}=T_1 (E_{i,j})$ for $1 \leq i \leq d-1$, $i+2 \leq j \leq d$, 
\\
$E_{1,i+1}=T_1(-E_{i,1})$ for $2 \leq i \leq d-1$, 
\\
$E_{3,2}-2E_{2,1}=T_1 (D_{1,2})$. 
\end{center}
Transposing, we also have:
\begin{center}
$\lbrace E_{i-1,j}-E_{i,j+1} : 1 \leq j \leq d-1$, $j+2 \leq i \leq d \rbrace $
$\cup \lbrace E_{j+1,1} : 2 \leq j \leq d-1 \rbrace 
\cup \lbrace E_{2,3}-2E_{1,2} \rbrace \subseteq im (T_2)$. 
\end{center}
It may therefore be seen that $im (T_1) \cup im (T_2)$ contains an $R$-basis for $\mathfrak{sl}_d (R)$, 
so $\mathfrak{sl}_d (R) = im (T_1) + im (T_2)$. 
Now suppose $d=2$ and $p > 2$. Then for any $a,b,c, \in R$, 
\begin{center}
$\left( \begin{array}{cc} a & b \\ c & -a \end{array} \right)
=(\left( \begin{array}{cc} 0 & -b \\ c & 0 \end{array} \right),\left( \begin{array}{cc} \frac{1}{2} & 0 \\ 0 & - \frac{1}{2} \end{array} \right))
+(\left( \begin{array}{cc} 0 & a \\ 0 & 0 \end{array} \right),\left( \begin{array}{cc} 0 & 0 \\ 1 & 0 \end{array} \right))$. 
\end{center}

\item[(c')] Let $\mathcal{B} = \lbrace x E_{i,j} : x \in R, i \neq j \rbrace \cup \lbrace x (D_{a,b}+E_{a,b}-E_{b,a}) : x \in R, a \leq b \rbrace$.\linebreak Then $\mathcal{B}$ clearly spans $\mathfrak{sl}_d (R)$ and,
for any $n \in \mathbb{N}$, $X \in \mathcal{B}$, \linebreak$det(I_d + \mathcal{P}^n X) = 1$. 
\end{itemize}

\begin{rmrk}
The preceding argument breaks down for $d = 2$, $p=2$. 
Let $X , Y \in \mathbb{M}_2 (R)$ with $tr(X) = tr(Y) = 0$. Then: 
\begin{center}
$(X,Y) \equiv (X_{12} Y_{21} - X_{21} Y_{12}) 
\left( \begin{array}{cc} 1 & 0 \\ 0 & -1 \end{array} \right) \text{ mod }\mathcal{P}$. 
\end{center}
Hence we cannot express an arbitrary traceless matrix as a sum of brackets, 
as we do above in higher characteristic. 
\end{rmrk}

\subsection{$SO_d$}

Denote by $\mathfrak{so}_d (R)$ the space of skew-symmetric $d \times d$ matrices over $R$; 
it is spanned over $R$ by the matrices $X_{i,j} = E_{i,j}-E_{j,i}$, for $1 \leq i < j \leq d$. 
\begin{itemize}
\item[(a')] Let $g \in K_n$. Write $g=I_d+\mathcal{P}^n X$, for some $X \in \mathbb{M}_d (R)$. Then:
\begin{center}
$I_d = (I_d + \mathcal{P}^n X)(I_d + \mathcal{P}^n X^T) 
\equiv I_d +\mathcal{P}^n (X+X^T) \text{ mod } \mathcal{P}^{2n}$
\end{center}
so $X^T \equiv -X \text{ mod }\mathcal{P}^n$. Hence there exists 
$X' \in \mathfrak{so}_d$ 
such that \linebreak$X \equiv X' \text{ mod }\mathcal{P}^n$. 

\item[(b')] Define the $R$-module endomorphisms 
$T_1 , T_2 , T_3 : \mathfrak{so}_d (R) \rightarrow \mathfrak{so}_d (R)$ by:
\begin{center}
$T_1 (X) = (X,\sum_{i=1} ^{d-1} X_{i,i+1})$, $T_2 (X) = (X,X_{1,d-1} + X_{1,d} + X_{2,d})$, $T_3 (X) = (X,X_{1,2})$. 
\end{center}
Then for $1 < i < d-1$, 
\begin{center}
$X_{i,i+2}-X_{i-1,i+1}=T_1 (X_{i,i+1})$; $X_{i+1,d}-X_{i-1,d}-X_{i,d-1}=T_1 (X_{i,d})$. 
\end{center}
For $1 < j < d-1$,
\begin{center}
$X_{2,j}+X_{1,j+1}-X_{1,j-1}=T_1 (X_{1,j})$. 
\end{center}
For $1<i,j<d$, with $i+1<j$, 
\begin{center}
$X_{i+1,j}+X_{i,j+1}-X_{i-1,j}-X_{i,j-1}$. 
\end{center}
For $3 \leq j \leq d-2$, 
\begin{center}
$X_{1,j}=T_2 (X_{j,d-1})$; $X_{j,d}=T_2 (-X_{2,j})$
\end{center}
and:
\begin{center}
$X_{1,2}=T_2 (-X_{2,d})$; $X_{d-1,d}=T_2 (X_{1,d-1})$; 
$X_{1,d-1}=T_3(-X_{2,d-1})$; $X_{1,d}=T_3(-X_{2,d})$; $X_{2,d}=T_3(-X_{1,d})$. 
\end{center}
Therefore $im (T_1) \cup im (T_2) \cup im (T_3)$ contains an $R$-basis for $\mathfrak{so}_d (R)$, 
so $\mathfrak{so}_d (R) = im (T_1) + im (T_2) + im (T_3)$. 

\item[(c')] For $\alpha \in R$, $l \in \mathbb{N}$, 
consider the polynomial $f(X) = X^2 - (1-\alpha^2 \mathcal{P}^{2l})$. \linebreak
Then $f(1)=\alpha^2 \mathcal{P}^{2l} \equiv 0 \text{ mod }\mathcal{P}^{2l}$ 
but $f'(1)=2 \not\equiv 0 \text{ mod } \mathcal{P}$. 
By Hensel's Lemma, there exists $\beta \in R$ such that $f(\beta)=0$ 
and \linebreak$\beta \equiv 1 \text{ mod } \mathcal{P}^{2l}$. 
Hence for any $i \neq j$, 
\begin{center}
$g_{i,j} ^{(l)} (\alpha)
 := I_d + \alpha \mathcal{P}^l (E_{i,j}-E_{j,i}) + (\beta-1) (E_{i,i}+E_{j,j}) \in K_l$ 
\end{center}
and $g_{i,j} ^{(l)} (\alpha) \equiv I_d + \alpha \mathcal{P}^l (E_{i,j}-E_{j,i}) \text{ mod }\mathcal{P}^{2l}$. 
\end{itemize}

\begin{rmrk}
In contrast to the cases of $SL_d (R)$ and $Sp_d (R)$, 
$SO_d (R)$ is not in general the universal form of the Chevalley group of its type; 
the universal form is rather a proper central extension of $SO_d (R)$ by a finite group. 
Increasing the constant $C_1$ in Theorem \ref{chevdiam}, 
the diameter bounds obtained above for $SO_d$ extend to the universal form by Lemma \ref{easyextnlem} (ii). 
\end{rmrk}

\subsection{$Sp_d$}

Let $d=2g$ and let $\Omega = \left( \begin{array}{cc} 0 & I_g \\ -I_g & 0 \end{array} \right)$, 
so that $\mathfrak{sp}_d (R)$ is the set of $d \times d$ matrices $X$ 
over $R$ satisfying the relation $X^T \Omega + \Omega X = 0$. Suppose $p > 2$. 

For $1 \leq i,j \leq g$, define the matrices:
\begin{center}
$A_{i,j}=\left( \begin{array}{cc} E_{i,j} & 0 \\ 0 & -E_{j,i} \end{array} \right)$ , 
$B_{i,j}=\left( \begin{array}{cc} 0 & E_{i,j}+E_{j,i} \\ 0 & 0 \end{array} \right)$ , 
$C_{i,j}=\left( \begin{array}{cc} 0 & 0 \\ E_{i,j}+E_{j,i} & 0 \end{array} \right) \in \mathfrak{sp}_d (R)$. 
\end{center}

We have:
\begin{center}
$(A_{i,j},A_{k,l}) = \delta_{j,k} A_{i,l} - \delta_{i,l} A_{k,j}$, \\
$(A_{i,j},B_{k,l}) = \delta_{j,k} B_{i,l} + \delta_{j,l} B_{i,k}$, \\
$(A_{i,j},C_{k,l}) = \delta_{i,l} C_{j,k} - \delta_{i,k} C_{j,l}$, \\
$(B_{i,j},C_{k,l}) = \delta_{j,k} A_{i,l} + \delta_{j,l} A_{i,k} + \delta_{i,k} A_{j,l} + \delta_{i,l} A_{j,k}$. 
\end{center}
Hence: 
\begin{center}
$A_{i,j} = (A_{i,j},A_{j,j})$, for $i \neq j$, \\
$A_{i,i} = (\frac{1}{2} B_{i,i},\frac{1}{2} C_{i,i})$, \\
$B_{i,j} = (A_{i,k},B_{k,j})$, for $i \neq k \neq j$, \\
$C_{i,j} = (A_{k,i},C_{j,k})$, for $i \neq k \neq j$. 
\end{center}

\begin{itemize}
\item[(a')] Let $g \in K_n$. Write $g=I_d+\mathcal{P}^n X$, 
for some $X \in \mathbb{M}_d (R)$. Then:
\begin{center}
$\Omega=g^T \Omega g=\Omega+\mathcal{P}^n (\Omega X+X^T\Omega)+\mathcal{P}^{2n} X^T\Omega X$ \linebreak
$\equiv \Omega+\mathcal{P}^n (\Omega X+X^T\Omega) \text{ mod } \mathcal{P}^{2n}$
\end{center}
so $\Omega X+X^T\Omega \equiv 0 \text{ mod } \mathcal{P}^n$. Hence there exists 
$X' \in \mathfrak{sp}_d (R)$ such that $X \equiv X' \text{ mod }\mathcal{P}^n$. 

\item[(b')] Define the $R$-module endomorphisms 
$U_1 , U_2 : \mathfrak{sp}_d (R) \rightarrow \mathfrak{sp}_d (R)$ by:
\begin{center}
$U_1 (X) = (X,\sum_{i=1} ^g A_{i,i})$, $U_2 (X) = (X,\sum_{i=1} ^g (B_{i,i}+C_{i,i}))$. 
\end{center}
Then for any $1 \leq i,j \leq g$, $B_{i,j},C_{i,j} \in im(U_1)$, $A_{i,j}+A_{j,i} \in im(U_2)$. 
Define the $R$-Lie subring $V \leq \mathfrak{sp}_d (R)$: 
\begin{center}
$V = \lbrace \left( \begin{array}{cc} X & 0 \\ 0 & X \end{array} \right) : X \in \mathfrak{gl}_g (R) \rbrace$. 
\end{center}
We show that, for any $X \in \mathfrak{so}_g (R)$,  
there exist $v_1 , v_2 \in V$ and symmetric $Z \in \mathfrak{gl}_g (R)$ such that: 
\begin{center}
$\left( \begin{array}{cc} X & 0 \\ 0 & X \end{array} \right) 
= (v_1,v_2) + \left( \begin{array}{cc} Z & 0 \\ 0 & -Z \end{array} \right)$. 
\end{center}
Now for an arbitrary element $v \in \mathfrak{sp}_d (R)$ there exist $X \in \mathfrak{so}_g (R)$, 
$B,C,Y \in \mathfrak{gl}_g (R)$ with $Y$ symmetric such that: 
\begin{center}
$v = \left( \begin{array}{cc} X & 0 \\ 0 & X \end{array} \right) 
+ \left( \begin{array}{cc} Y & 0 \\ 0 & -Y \end{array} \right) 
+ \left( \begin{array}{cc} 0 & B \\ C & 0 \end{array} \right)$
$=(v_1,v_2) + \left( \begin{array}{cc} 0 & B \\ C & 0 \end{array} \right) 
+ \left( \begin{array}{cc} Y+Z & 0 \\ 0 & -(Y+Z) \end{array} \right)$
\end{center}
and $\left( \begin{array}{cc} 0 & B \\ C & 0 \end{array} \right) \in im (U_1)$, 
$\left( \begin{array}{cc} Y+Z & 0 \\ 0 & -(Y+Z) \end{array} \right) \in im (U_2)$, 
so that every element of $\mathfrak{sp}_d (R)$ is expressible as a sum of three brackets. 
\\ \\
It will suffice to check that any element of $\mathfrak{so}_g (R)$ 
is expressible as the sum of a bracket in $\mathfrak{gl}_g (R)$ and a symmetric matrix. 
Define the $R$-module endomorphisms $S_1 , S_2 : \mathfrak{gl}_g (R) \rightarrow \mathfrak{gl}_g (R)$ by: 
\begin{center}
$S_1 (X) = (X,E_{1,1})$; $S_2 (X) = (X,\sum_{i=1} ^{d-1} (E_{i,i+1}-E_{i+1,i}))$
\end{center}
and for $X \in \mathfrak{gl}_g (R)$, write $X=X_1 + X_2$, 
with $X_1$ symmetric, $X_2$ skew-symmetric. Then:
\begin{center}
$(X,E_{1,1}+\sum_{i=1} ^{d-1} (E_{i,i+1}-E_{i+1,i}) - S_1 (X_1) - S_2 (X_2)$
\end{center}
is symmetric. We already described the image of $S_2 \mid_{\mathfrak{so}_g (R)}$ 
(in the guise of $T_1$ in our analysis of $SO_d$). For $2 \leq i \leq g$, 
\begin{center}
$S_1 (E_{1,i}+E_{i,1}) = E_{i,1}-E_{1,i}$. 
\end{center}
These elements, together with $im(S_2 \mid_{\mathfrak{so}_g (R)})$, span $\mathfrak{so}_g (R)$ over $R$, 
and the result follows. 

\item[(c')] For any $\alpha \in R$ and any $l \in \mathbb{N}$ we have:
\begin{center}
$I_d + \alpha \mathcal{P}^l B_{i,j} , I_d + \alpha \mathcal{P}^l C_{i,j} \in K_l$ for any $1 \leq i,j, \leq d$. 
\\
$I_d + \alpha \mathcal{P}^l A_{i,j} \in K_l$ provided $i \neq j$. 
\end{center}
Finally, 
$(1 + \alpha \mathcal{P}^l)^{-1} \equiv 1 - \alpha \mathcal{P}^l \text{ mod } \mathcal{P}^{2l}$, so:
\begin{center}
$K_l \ni 
I + \left( \begin{array}{cc} \alpha \mathcal{P}^l E_{i,i} & 0 \\ 0 & ((1 + \alpha \mathcal{P}^l)^{-1} -1) E_{i,i} \end{array} \right) 
\equiv I + \alpha t^\mathcal{P} A_{i,i} \text{ mod }\mathcal{P}^{2l}$. 
\end{center}
\end{itemize}

\begin{rmrk}
For $R=\mathbb{Z}_p$, the value $A=3$ was achieved in \cite{Dinai}, 
under the additional assumption that $p \geq \frac{l+2}{2}$, 
where $l$ is the rank of the associated Chevalley group scheme. 
This assumption was necessary in the specific manipulations the root systems which were applied in Dinai's argument. 
Hence even in the $p$-adic case, the results of this Section are new in large rank for small $p$. 
\end{rmrk}

\section{Analytic Pro-$p$ Groups} \label{AnalyticSection}

In this section we prove Theorem \ref{Rstandarddiam}. 
We start by recalling some preliminaries about groups with an $R$-analytic structure. 
Recall that $(R,\mathcal{M})$ is a discrete valuation pro-$p$ domain, 
with $\mathcal{M}$ generated by $\mathcal{P} \in \mathcal{M}$. 
For proofs of results quoted, refer to Chapter 13 of \cite{DiDuMaSe}. 

\begin{defn}
Denote by $R[[\underline{X},\underline{Y}]]$ the ring of formal non-commuting power series in the $2d$ variables 
$X_1 , \ldots , X_d , Y_1 , \ldots , Y_d$. 
For $i=1 , \ldots , d$, let $F_i (\underline{X},\underline{Y}) \in R[[\underline{X},\underline{Y}]]$. 
Then $\underline{F} = (F_1 , \ldots , F_d)$ is a \emph{formal group law}, 
of dimension $d$ over $R$, if: 
\begin{itemize}
\item[(i)] $\underline{F}(\underline{X},\underline{0}) = \underline{X}$ 
and $\underline{F}(\underline{0},\underline{Y}) = \underline{Y}$, 
\item[(ii)] $\underline{F}(\underline{X},\underline{F}(\underline{Y},\underline{Z})) 
= \underline{F}(\underline{F}(\underline{X},\underline{Y}),\underline{Z})$. 
\end{itemize}
\end{defn}

\begin{propn}[13.16 in \cite{DiDuMaSe}] \label{pwrseriesid}
Let $\underline{F}$ be a formal group law. 
There exist power series $\underline{B}(\underline{X},\underline{Y})$, $\underline{I}(\underline{X})$, 
$\underline{O}(\underline{X},\underline{Y})$, $\underline{P}(\underline{X})$, 
$\underline{Q}(\underline{X},\underline{Y})$, 
with $\underline{B}$ bilinear in $\underline{X}$ and $\underline{Y}$; 
every term of $\underline{O},\underline{P},\underline{Q}$ 
having total degree at least $3$ and every term of $\underline{O},\underline{Q}$ 
having degree at least $1$ in each of $\underline{X},\underline{Y}$, such that: 
\begin{itemize}
\item[(i)] $\underline{F}(\underline{X},\underline{Y})
=\underline{X}+\underline{Y}+\underline{B}(\underline{X},\underline{Y})+\underline{O}(\underline{X},\underline{Y})$, 
\item[(ii)] $\underline{I}(\underline{X})=-\underline{X}+\underline{B}(\underline{X},\underline{X})+\underline{P}(\underline{X})$
and $\underline{F}(\underline{X},\underline{I}(\underline{X}))=\underline{0}
=\underline{F}(\underline{I}(\underline{X}),\underline{X})$, 
\item[(iii)] $\underline{F}((\underline{I} \circ \underline{F})(\underline{Y},\underline{X}),\underline{F}(\underline{X},\underline{Y}))
=\underline{B}(\underline{X},\underline{Y})-\underline{B}(\underline{Y},\underline{X})
+\underline{Q}(\underline{X},\underline{Y})$. 
\end{itemize}
\end{propn}

\begin{defn}
An \emph{$R$-standard group} of dimension $d$ is a topological group $(G, \cdot)$ 
with underlying space $G = \mathcal{M}^{(d)}$ 
such that there exists a formal group law $\underline{F}$ of dimension $d$ such that, for all $g,h \in G$, 
\begin{center}
$g \cdot h = \underline{F}(g,h)$. 
\end{center}
Note that, for $\underline{B},\underline{I},\underline{Q}$ as in Proposition \ref{pwrseriesid}, we have: 
\begin{center}
$g^{-1} = \underline{I}(g)$, $[g,h] = \underline{B}(g,h)-\underline{B}(h,g)+\underline{Q}(g,h)$. 
\end{center}
\end{defn}

\begin{ex}
\begin{itemize}
\item[(i)] $(\mathcal{M}^{(d)},+)$ is an $R$-standard group of dimension $d$. 

\item[(ii)] Let $GL_d ^1 (R) = I_d + \mathcal{P} \mathbb{M}_d (R)$. 
Then $GL_d ^1 (R) \leq GL_d (R)$ and, identifying $GL_d ^1 (R)$ 
with $\mathcal{M}^{(d^2)}$ in the obvious way, 
multiplication in $GL_d ^1 (R)$ is given by a formal group law of dimension $d^2$. 

\item[(iii)] Let $SL_d ^1 (R) = SL_d (R) \cap GL_d ^1 (R)$ be the kernel 
of the congruence map $SL_d (R) \twoheadrightarrow SL_d (R / \mathcal{M})$. 
Then we may identify $SL_d ^1 (R)$ with $\mathcal{M}^{(d^2 - 1)}$ 
via $A \mapsto ((A-I_d)_{i,j})_{(i,j)\neq(d,d)}$ 
(since these $d^2-1$ co-ordinates together with the determinant condition uniquely determine $A_{d,d}$). 
Under this identification, multiplication in $SL_d ^1 (R)$ is given by a formal group law of dimension $d^2 - 1$. 
\end{itemize}
\end{ex}

\begin{propn}[13.22 in \cite{DiDuMaSe}] \label{propfiltration}
For $n,m \in \mathbb{N}$, let $K_n = (\mathcal{M}^n)^{(d)} \subseteq G$. Then:
\begin{itemize}
\item[(i)] $K_n \vartriangleleft_o K_1 = G$, 
\item[(ii)] $[K_n , K_m] \subseteq K_{n+m}$, 
\item[(iii)] If $m \leq n$, $K_n / K_{n+m}$ 
is isomorphic to the additive group $(\mathcal{M}^n / \mathcal{M}^{n+m})^{(d)}$ 
\item[(iv)] $G \cong \varprojlim G / K_n$ is a pro-$p$ group. 
\end{itemize}
\end{propn}

\begin{thm}[13.20 in \cite{DiDuMaSe}]
Let $G$ be an $R$-analytic group. Then $G$ has an open $R$-standard subgroup. 
\end{thm}

\begin{propn}[13.24 in \cite{DiDuMaSe}]
For $v,w \in \mathcal{M}^{(d)}$, define: 
\begin{center}
$(v,w)=\underline{B}(v,w)-\underline{B}(w,v)$. 
\end{center}
Then $L(G)=(\mathcal{M}^{(d)},+,(\cdot,\cdot))$ is a $R$-Lie ring. 
That is, $(\cdot,\cdot)$ satisfies the Jacobi identity (and is obviously $R$-bilinear antisymmetric). 
\end{propn}

\begin{rmrk}
For each $n$, $\mathcal{P}^n L(G)$ is a Lie subring of $L(G)$. As a set it is equal to $K_{n+1}$. 
Moreover by Proposition \ref{pwrseriesid}, the additive cosets of $\mathcal{P}^n L(G)$ in $L(G)$ 
are the same as the multiplicative cosets of $K_{n+1}$ in $G$. 
\end{rmrk}

\begin{defn} \label{LieAlgdefn}
The \emph{Lie algebra} of $G$ is $\mathcal{L}_{G} = L(G) \bigotimes_{R} \mathbb{K}$, 
where $\mathbb{K}$ is the field of fractions of $R$. 
\end{defn}

\begin{ex}
\begin{itemize}
\item[(i)] For $G = (\mathcal{M}^{(d)},+)$, 
$\mathcal{L}_G$ is the $d$-dimensional abelian $\mathbb{K}$-Lie algebra. 

\item[(ii)] $\mathcal{L}_{GL_d ^1 (R)} = \mathfrak{gl}_d (\mathbb{K})$. 

\item[(iii)] $\mathcal{L}_{SL_d ^1 (R)} = \mathfrak{sl}_d (\mathbb{K})$. 

\end{itemize}
\end{ex}

\begin{propn} \label{Fabbrack}
Suppose $\mathcal{L}_G$ is perfect. 
There exists $k \in \mathbb{N}$ such that every element of $(\mathcal{M}^k)^{(d)}$ 
is expressible as a sum of at most $d$ brackets in $L_G$. 
\end{propn}

\begin{proof}
Let $\lbrace x_1 , \ldots , x_d \rbrace$ be a $R$-basis for $L(G)$. 
Then there exist \linebreak$r_i , s_i \in \lbrace 1 , \ldots , d \rbrace$ 
such that $\lbrace (x_{r_1},x_{s_1}) , \ldots , (x_{r_d},x_{s_d}) \rbrace$ 
is a $\mathbb{K}$-basis for $\mathcal{L}_G$. 
Let $\lambda_{i,j} \in \mathbb{K}$ be such that:
\begin{center}
$x_i = \sum_{j=1} ^d \lambda_{i,j} (x_{r_j},x_{s_j})$. 
\end{center}
Let $k \in \mathbb{N}$ be defined by:
\begin{center}
$\lVert \mathcal{P} \rVert^{-k} 
= max (\lbrace 1 \rbrace \cup \lbrace \lVert \lambda_{i,j} \rVert : 1 \leq i , j \leq d \rbrace)$. 
\end{center}
Then for any $1 \leq i , j \leq d$, $\mathcal{P}^k \lambda_{i,j} \in R$. 
Hence for any $x \in L(G)$, there exist $\mu_1 , \ldots , \mu_d \in R$ such that: 
\begin{center}
$\mathcal{P}^k x 
= \mathcal{P}^k \sum_{i=1} ^d \mu_i x_i 
= \mathcal{P}^k \sum_{i=1} ^d \mu_i \sum_{j=1} ^d \lambda_{i,j} (x_{r_j},x_{s_j})$
$= \sum_{j=1} ^d (\sum_{i=1} ^d \mu_i \mathcal{P}^k \lambda_{i,j} x_{r_j},x_{s_j})$
\end{center}
as required. 
\end{proof}

\begin{proof}[Proof of Theorem \ref{Rstandarddiam}]
We verify the hypotheses of Proposition \ref{SolovayKitaev2}. 
Hypothesis (i) is Proposition \ref{propfiltration} (ii). 
For hypothesis (ii), we take $\epsilon$ arbitrary; $A=d$; 
$M_1 \geq max \lbrace \frac{k}{3} + 1 , 2 \rbrace$; $M_2 = k$, 
where $k$ is as in Proposition \ref{Fabbrack}. 
For $i=1,2,3$ choose $\frac{n}{3} (2+i+ \epsilon) \leq n_i \leq m_i \leq \frac{2n}{3} (2+i)$ 
such that $n_i + m_i = (2+i)n-M_2$ (this is possible by our choice of $M_1 , M_2$). 
\\ \\
Let $g \in K_{(2+i)n}$. 
Let $h \in K_{M_2}$ be such that $g = \mathcal{P}^{n_i + m_i} h$. 
Then there exist $g_1 , \ldots , g_d , h_1 , \ldots , h_d \in G$ such that:
\begin{center}
$h = \sum_{i=1} ^d (g_i , h_i)$
\end{center}
so that: 
\begin{center}
$g = \sum_{i=1} ^d (\mathcal{P}^{n_i} g_i ,\mathcal{P}^{m_i} h_i)$
\\
$\equiv \sum_{i=1} ^d [\mathcal{P}^{n_i} g_i ,\mathcal{P}^{m_i} h_i] \text{ mod } \mathcal{P}^{2n_i+m_i}$ 
(by Proposition \ref{pwrseriesid} (iii))
\\
$\equiv [\mathcal{P}^{n_i} g_1 ,\mathcal{P}^{m_i} h_1] \cdots [\mathcal{P}^{n_i} g_d ,\mathcal{P}^{m_i} h_d] 
\text{ mod } \mathcal{P}^{2n_i+2m_i}$ (by Proposition \ref{pwrseriesid} (i)). 
\end{center}
Since $2n_i+m_i = (2+i)n + n_i - M_2$, we are done. 
\end{proof}

\subsection{FAb $p$-adic Analytic Groups} \label{padicanalyticSection}

In the case of a group $G$ with an analytic structure over $\mathbb{Z}_p$, 
there is an alternative approach to constructing the Lie algebra of $G$, 
based on the concept of a \emph{uniform} subgroup, 
rather than a \emph{$\mathbb{Z}_p$-standard} subgroup. 
We will utilise this approach to complete the proof of Theorem \ref{padicdiam}. 
Let $p \geq 3$ be prime. Let $G$ be a finitely generated pro-$p$ group. 
Let $(G_n)_n$ be the lower central $p$-series of $G$. 

\begin{defn}
$G$ is \emph{powerful} if $G / \overline{G^p}$ is abelian. 
$G$ is \emph{uniform} if it is powerful and torsion-free. 
The \emph{dimension} of a uniform group $G$ is the minimal size of a topological generating set. 
\end{defn}

\begin{ex} \label{standarduniform}
Recall (\cite{Laz}) that every compact $p$-adic analytic group has an open characteristic uniform subgroup. 
Indeed, every $\mathbb{Z}_p$-standard group of dimension $d$ is a uniform pro-$p$ group of dimension $d$ 
(8.31 of \cite{DiDuMaSe}). Conversely, if $G$ is a $d$-dimensional uniform pro-$p$ group, 
then $G_2$ is a $d$-dimensional $\mathbb{Z}_p$-standard group (8.23 (iii) of \cite{DiDuMaSe}). 
In particular, every compact $p$-adic analytic group has an open characteristic $\mathbb{Z}_p$-standard subgroup. 
We describe the formal group law on $G_2$ below. 
\end{ex}

We recall some properties of uniform groups.
Unless otherwise specified, let $G$ be a $d$-dimensional uniform group. 

\begin{thm}[3.6, 4.9 in \cite{DiDuMaSe}] \label{padicfiltgen}
Let $\lbrace a_1 , \ldots , a_d \rbrace$ be a topological generating set for $G$; $n , m \in \mathbb{N}$. 
\begin{itemize}
\item[(i)] $(\lambda_1 , \ldots , \lambda_d) \mapsto a_1 ^{\lambda_1} \cdots a_d ^{\lambda_d}$ 
defines a homeomorphism $\mathbb{Z}_p ^d \rightarrow G$. 
\item[(ii)] $G_{n+1}$ is uniform of dimension $d$. 
\item[(iii)] $(G_{n+1})_{m+1} = G_{m+n+1}$. 
\item[(iv)] $G_{n+1} = \lbrace x^{p^n} : x \in G \rbrace$. 
\item[(v)] $\lbrace a_1 ^{p^n} , \ldots , a_d ^{p^n} \rbrace$ is a topological generating set for $G_{n+1}$
\end{itemize}
\end{thm}

There is a complete normed $\mathbb{Q}_p$-algebra $\hat{A}$, an embedding $G \hookrightarrow \hat{A}^*$ satisfying: 
\begin{center}
$\forall g \in G, g-1 \in \hat{A}_0$, where $\hat{A}_0 = \lbrace x \in \hat{A} : \lVert x \rVert \leq p^{-1} \rbrace$
\end{center}
and mutually inverse analytic functions:
\begin{center}
$log : 1 + \hat{A}_0 \rightarrow \hat{A}_0$, 
\\
$exp : \hat{A}_0 \rightarrow 1 + \hat{A}_0$. 
\end{center}
$\hat{A}$ is naturally a $\mathbb{Q}_p$-Lie algebra with Lie bracket:
\begin{center}
$(x,y)=xy-yx$. 
\end{center}
$log(G)$ is a free $d$-dimensional $\mathbb{Z}_p$-module and a $\mathbb{Z}_p$-Lie subalgebra of $\hat{A}$. 

\begin{lem}[6.25 and 7.12 from \cite{DiDuMaSe}]
Let $x \in \hat{A}_0$, $n \in \mathbb{Z}$. 
\begin{itemize}
\item[(i)] $exp(nx) = exp(x)^n$. 
\item[(ii)] $log((1+x)^n) = n log(1+x)$. 
\item[(iii)] $(log(G),log(G)) \subseteq p log(G)$. 
\end{itemize}
Moreover, for $g \in G$, $\lambda \in \mathbb{Z}_p$, $\lambda log(g) = log(g^{\lambda})$. 
\end{lem}

Combining this Lemma with Theorem \ref{padicfiltgen} (iv), we have: 

\begin{coroll} \label{ballsubgrpcoroll}
For all $n \in \mathbb{N}$, $p^n log(G) = log(G_{n+1})$. 
\end{coroll}

\begin{propn}[6.27 and 6.28 in \cite{DiDuMaSe}] \label{seriesops}
There are formal non-commutative power series $\Phi (X,Y)$, $\Psi (X,Y)$ satisfying:
\begin{center}
$\Phi (X,Y) = X+Y+\frac{1}{2} (XY-YX)+h.o.(X,Y)$
\\
$\Psi (X,Y) = (XY-YX)+h.o.(X,Y)$
\end{center}
(with $h.o.(X,Y)$ denoting terms composed of brackets of length at least three) 
such that, for $x , y \in \hat{A}_0$, 
\begin{itemize}
\item[(i)] $\Phi (x,y)$ converges to $log (exp(x) exp(y))$, 
\item[(ii)] $\Psi (x,y)$ converges to $log (exp(-x) exp(-y) exp(x) exp(y))$. 
\end{itemize}
\end{propn}

\begin{rmrk} \label{uniformgivesstandard}
Let $x_1 , \ldots , x_d \in log(G)$ be a $\mathbb{Z}_p$-basis for $log(G)$. 
Identify $\mathbb{Z}_p^{(d)}$ with $log(G)$ via:
\begin{center}
$\theta : (\alpha_i)_{i=1} ^d \mapsto \sum_{i=1} ^d \alpha_i x_i$. 
\end{center}
Then, identifying $\mathbb{Z}_p^{(d)}$ with $G$ via $exp \circ \theta$, 
multiplication in $G$ corresponds to the formal group law: 
\begin{center}
$(\underline{a},\underline{b}) \mapsto \theta^{-1} (\Phi(\theta(\underline{a}),\theta(\underline{b})))$
\end{center}
on $\mathbb{Z}_p ^{(d)}$. Moreover, under this identification the subgroup $G_{n+1}$ 
corresponds to $p^n log(G) = \theta ((p^n \mathbb{Z}_p)^{(d)})$, by Corollary \ref{ballsubgrpcoroll}. 
In particular, $G_2 \cong (p \mathbb{Z}_p)^{(d)}$ is a $\mathbb{Z}_p$-standard subgroup. 
\end{rmrk}

\begin{propn}[4.8 and 4.31 in \cite{DiDuMaSe}] \label{dimsgrpsalgs}
Let $H$ be a uniform closed subgroup of $G$; $N \vartriangleleft G$ be closed such that $G/N$ is uniform. 
\begin{itemize}
\item[(i)] $log(H)$ is a $\mathbb{Z}_p$-subalgebra of $log(G)$. 
\item[(ii)] $N$ is uniform, with $dim(N)=dim(G)-dim(G/N)$. 
\item[(iii)] $log(N)$ is an ideal in $log(G)$, and $log(G/N) \cong log(G)/log(N)$. 
\end{itemize}
\end{propn}

\begin{propn}[7.15 in \cite{DiDuMaSe}] \label{dimsalgsgrps}
Let $S$ be a $\mathbb{Z}_p$-Lie subalgebra of $log(G)$ 
such that the $\mathbb{Z}_p$-module $log(G)/S$ is torsion-free. 
\begin{itemize}
\item[(i)] $exp(S)$ is a closed uniform subgroup of $G$. 
\item[(ii)] If $S$ is an ideal of $log(G)$, then $exp(S) \vartriangleleft G$ and $G/exp(S)$ is uniform. 
\end{itemize}
\end{propn}

Define $\mathcal{L}_G = span_{\mathbb{Q}_p} (log(G))$, a $d$-dimensional $\mathbb{Q}_p$-Lie algebra. 
By Remark \ref{uniformgivesstandard}, this is isomorphic to the Lie algebra described in Definition \ref{LieAlgdefn}. 

\begin{propn} \label{Fabalg}
The following are equivalent:
\begin{itemize}
\item[(i)] $G$ has finite abelianisation. 
\item[(ii)] $G$ is FAb. 
\item[(iii)] $\mathcal{L}_G$ is perfect. 
\end{itemize}
\end{propn}

\begin{proof}
$(ii) \Rightarrow (i)$ is clear. 
\\ \\
For $(iii) \Rightarrow (ii)$, suppose $H \leq_o G$ is such that  $\exists \phi :H \twoheadrightarrow \mathbb{Z}_p$. We may suppose that $H = G^{p^n}$ for some $n \in \mathbb{N}$. For if $h \in H$ is such that $\mathbb{Z}_p = \overline{\langle \phi(h) \rangle}$, and $n \in \mathbb{N}$ is such that $G^{p^n} \leq H$, then $h^{p^n} \in G^{p^n}$, and $p^n \mathbb{Z}_p = \overline{\langle \phi(h^{p^n}) \rangle} \leq \phi(G^{p^n}) \leq \mathbb{Z}_p$ , so $\phi(G^{p^n}) \leq_o \mathbb{Z}_p$, and $\phi(G^{p^n}) \cong \mathbb{Z}_p$. 
\\ \\
Now let $N = ker (\phi)$, so that by Proposition \ref{dimsgrpsalgs} (ii), 
$N \vartriangleleft_c H$ is uniform of dimension $d-1$; 
$log(H)=p^n log(G)$ and $log(H)/log(N) \cong \mathbb{Z}_p$. 
\\ \\
Hence $\mathcal{L}_H = \mathcal{L}_G$ so $\mathcal{L}_G / \mathcal{L}_N \cong \mathbb{Q}_p$, 
and $\mathcal{L}_G$ is not perfect. 
\\ \\
For $(i) \Rightarrow (iii)$, suppose $\mathcal{I} \vartriangleleft \mathcal{L}_G$, with $dim(\mathcal{I}) = d - 1$. 
Let \linebreak$I = log(G) \cap \mathcal{I} \vartriangleleft log(G)$ (so that $\mathcal{I} = span_{\mathbb{Q}_p}(I)$). 
Let $v \in log(G)$, 
and suppose $\exists \lambda \in \mathbb{Z}_p \setminus \lbrace 0 \rbrace$ such that $\lambda v \in I$. 
Then $v = \lambda^{-1} (\lambda v) \in \mathcal{I}$, 
so $v \in \mathcal{I} \cap log(G) = I$. 
Thus $log(G) / I$ is torsion-free, so by Proposition \ref{dimsalgsgrps}, $exp(I) \vartriangleleft G$ is uniform and $G/exp(I)$ is uniform, with:
\begin{center}
 $dim(G/exp(I)) = dim(G) - dim(exp(I))$ $= rk(log(G)) - rk(I) = dim(\mathcal{L}_G) - dim(\mathcal{I}) = 1$. 
\end{center}
and a $1$-dimensional uniform group is by definition infinite procyclic, so $G/exp(I) \cong \mathbb{Z}_p$. 
\end{proof}

\begin{proof}[Proof of Theorem \ref{padicdiam}] 
First suppose that $G$ is a FAb compact $p$-adic analytic group. 
As noted in Example \ref{standarduniform}, $G$ has an open characteristic uniform subgroup $H$. 
By Remark \ref{uniformgivesstandard}, $H_2$ is $\mathbb{Z}_p$-standard. 
Let $K_n \vartriangleleft_o H_2$ be as in Theorem \ref{Rstandarddiam}. 
Then by Remark \ref{uniformgivesstandard} and Theorem \ref{padicfiltgen} (iii), 
\begin{center}
$K_n = (H_2)_{n} = H_{n+1}$ 
\end{center}
and $H_{n+1}$ is a characteristic subgroup of $H$. In particular, $K_n \vartriangleleft_o G$. 
As in the proof of Theorem \ref{Rstandarddiam}, 
$(K_n)_n$ satisfies the hypotheses of Proposition \ref{SolovayKitaev2} and the result follows. 
\\ \\
Now suppose that $G$ is uniform and not FAb. By Proposition \ref{Fabalg}, \linebreak
$\exists \phi : G \twoheadrightarrow \mathbb{Z}_p$. 
By Proposition \ref{dimsgrpsalgs}, $N = ker (\phi)$ is uniform of dimension $d-1$. 
We may therefore choose a generating set $S = \lbrace a_1 , \ldots , a_d \rbrace$ for $G$ 
such that $\lbrace a_1 , \ldots , a_{d-1} \rbrace$ is a generating set for $N$ 
and $\overline{ \langle \phi (a_d) \rangle } = \mathbb{Z}_p$. 
\\ \\
Let $\pi_n : \mathbb{Z}_p \twoheadrightarrow \mathbb{Z} / p^n \mathbb{Z}$ 
be the natural projection. Then $G_{n+1} \subseteq ker (\pi_n \circ \phi)$, so:
\begin{center}
$diam(G / G_{n+1} , S) \geq diam(\mathbb{Z} / p^n \mathbb{Z},\lbrace(\pi_n \circ \phi)(a_d)\rbrace) \geq C p^n 
= C \lvert G / G_{n+1} \rvert^{\frac{1}{d}}$
\end{center}
In particular $diam(G / G_{n+1} , S)$ is not polylogarithmic in $\lvert G / G_{n+1} \rvert$. 
\end{proof}

\subsection{Exceptional Groups over $R$}

With Theorem \ref{Rstandarddiam} in hand we may complete the proof of Theorem \ref{chevdiam}. 
We start by marshalling some facts about Chevalley groups. 
Unless otherwise stated, proofs of assertions left unproven in this section may be found in \cite{Carter}. 
\\ \\
Let $\Phi$ be a root system of type $X_l \in \lbrace A_l,B_l,C_l,D_l,E_6,E_7,E_8,F_4,G_2 \rbrace$, 
$\Pi \subseteq \Phi$ be a fundamental system of roots, 
and $S$ be a commutative unital ring. We define the \emph{universal Chevalley group of type $X_l$ over $S$} 
to be the group $\mathcal{G}_S (X_l)$ abstractly generated by the symbols 
$\lbrace x_{\alpha} (t) \rbrace_{\alpha \in \Phi ; t \in S}$, 
subject to the \emph{Steinberg relations}. 
These are described in detail in \cite{Carter}; 
the only fact we require about them is: 
\begin{itemize}
\item[(a)] If $S'$ is a subring of $S$, 
then the inclusion of $\lbrace x_{\alpha} (t) \rbrace_{\alpha \in \Phi ; t \in S'}$ 
into \linebreak$\lbrace x_{\alpha} (t) \rbrace_{\alpha \in \Phi ; t \in S}$ 
induces a homomorphism $\psi : \mathcal{G}_{S'} (X_l) \rightarrow \mathcal{G}_S (X_l)$ 
(in other words, for each $\Phi$, every Steinberg relation over $S'$ is also a Steinberg relation over $S$). 
\end{itemize}
We define, for each $\alpha \in \Phi , s \in S^*$, the element:
\begin{center}
$c_{\alpha}(s)=x_{\alpha}(s)x_{-\alpha}(-s^{-1})x_{\alpha}(s)(x_{\alpha}(1)x_{-\alpha}(-1)x_{\alpha} (1))^{-1}$. 
\end{center}
Trivially,
\begin{itemize}
\item[(b)] $c_{\alpha}(1) = e$. 
\end{itemize}

\begin{thm}[Exercise 13.11 in \cite{DiDuMaSe}] \label{univchev}
Let $(R,\mathcal{M})$ be a pro-$p$ domain. For each $n \geq 1$, 
let $G_n \leq \mathcal{G}_{R} (X_l)$ be the subgroup generated by the set:
\begin{center}
$\lbrace x_{\alpha}(t) \rbrace_{\alpha \in \Phi ; t \in \mathcal{M}^n} 
\cup \lbrace c_{\alpha}(1+s) \rbrace_{\alpha \in \Phi ; s \in \mathcal{M}^n}$. 
\end{center}
Then:
\begin{itemize}
\item[(i)] $G_n \vartriangleleft_f \mathcal{G}_{R} (X_l)$, for all $n \geq 1$. 

\item[(ii)] The map $\theta_n : (\mathcal{M}^n)^{(\lvert \Phi \rvert+\lvert \Pi \rvert)} \rightarrow G_n$, given by: 
\begin{center}
$\theta (\underline{t}) = (\prod_{\alpha \in \Phi^+} x_{\alpha}(t_{\alpha})) 
(\prod_{\alpha \in \Pi} c_{\alpha} (1+t_{\alpha})) 
(\prod_{\alpha \in \Phi^-} x_{\alpha}(t_{\alpha}))$
\end{center}
(with the products ordered by the height function induced on $\Phi$ by $\Pi$) is a bijection, for every $n \geq 1$. 
Identifying $G_1$ with $\mathcal{M}^{(\lvert \Phi \rvert+\lvert \Pi \rvert)}$ via $\theta_1$, 
$G_1$ is an $R$-standard group of dimension $\lvert \Phi \rvert+\lvert \Pi \rvert$. 

\item[(iii)] $\mathcal{L}_{G_1}$ is perfect, unless $p=2$ and $X_l = A_1$ or $C_l$. 
Indeed $\mathcal{L}_{G_1}$ is the $\mathbb{K}$-Lie algebra of type $X_l$. 
\end{itemize}
\end{thm}

For $K$ a field, $\mathcal{G}_{K} (X_l)$ acts on the $K$-Lie algebra 
$\mathcal{L}_{K} (X_l)$ of type $X_l$ by linear automorphisms. 
For $\lbrace E_{\alpha} \rbrace_{\alpha \in \Phi} \cup \lbrace H_{\beta} \rbrace_{\beta \in \Pi}$ 
a Chevalley basis for $\mathcal{L}_{K} (X_l)$, the action may be defined by: 
\begin{itemize}
\item[(i)] $x_{\alpha} (t) (E_{\alpha}) = E_{\alpha}$

\item[(ii)] $x_{\alpha} (t) (E_{-\alpha}) = E_{-\alpha} + t H_{\alpha} - t^2 E_{\alpha}$

\item[(iii)] $x_{\alpha (t)} (H_{\alpha}) = H_{\alpha} - 2 t E_{\alpha}$

\item[(iv)] $x_{\alpha (t)} (H_{\beta}) = H_{\beta} - A_{\beta,\alpha} t E_{\alpha}$

\item[(v)] $x_{\alpha (t)} (E_{\beta}) = E_{\beta} + \sum_{i=1} ^q M_{\alpha,\beta,i} t^i E_{i \alpha + \beta}$

\end{itemize}
for any $\alpha , \beta \in \Phi$ linearly independent and $t \in K$. 
Here $A_{\beta,\alpha} = \frac{2(\beta,\alpha)}{(\alpha,\alpha)}$ is the \emph{Cartan integer}; 
$M_{\alpha,\beta,i}$ are integers and $q \in \mathbb{N}$ is maximal such that $q \alpha + \beta \in \Phi$. 
\\ \\
Now take $K = \mathbb{K}$, the field of fractions of $R$. 
Let $\rho : \mathcal{G}_{\mathbb{K}} (X_l) \rightarrow GL_d (\mathbb{K})$ 
be the above-described action (where d = $\lvert \Phi \rvert+\lvert \Pi \rvert$ 
is the dimension of $\mathcal{L}_{K} (X_l)$). 
Let $\psi : \mathcal{G}_{R} (X_l) \rightarrow \mathcal{G}_{K} (X_l)$ 
be as described in observation (a). 
The \emph{adjoint Chevalley group of type $X_l$ over $R$} 
is defined to be the group $G_{ad} = \rho(\psi(\mathcal{G}_{R} (X_l)))$. 
It is clear from (i)-(v) above that: 
\begin{itemize}
\item[(c)] $G_{ad} \leq GL_d (R)$. 

\item[(d)] For any $\alpha \in \Phi; s,t \in R$ and $n \geq 1$, if $s \equiv t \text{ mod } \mathcal{M}^n$ then 
\linebreak$\rho(x_{\alpha}(s)) \equiv \rho(x_{\alpha}(t)) \text{ mod } \mathcal{M}^n$. 

\item[(e)] In particular, for $t \in \mathcal{M}^n$, $\rho(x_{\alpha}(t)) \equiv I_d \text{ mod } \mathcal{M}^n$. 
\end{itemize}
From observations (b) and (d), it follows that for any $\beta \in \Phi,s \in \mathcal{M}^n$, 
$\rho(c_{\beta}(1+s)) \equiv I_d \text{ mod } \mathcal{M}^n$. 
Combining with observation (e), we have:

\begin{equation} \label{exceptionalquotient}
\rho(\psi(G_n)) \leq K_n := G_{ad} \cap (I_d + \mathbb{M}_d (\mathcal{M}^n)). 
\end{equation}

\begin{proof}[Proof of Theorem \ref{chevdiam}]
If $X_l \in \lbrace A_l , B_l , C_l , D_l \rbrace$, 
$G_{ad}$ is one of $PSL_d (R)$, \linebreak$PSO_d (R)$ or $PSp_{d} (R)$. 
The result then follows as in Section \ref{classicalsection}. 
If not, then letting $G_1$ be as in Theorem \ref{univchev}, 
$G_1$ satisfies the hypothesis of Theorem \ref{Rstandarddiam}, 
so that for some $\tilde{C}_1 , C_2 > 0$, 
\begin{center}
$diam(\mathcal{G}_{R} (X_l)/G_n) \leq \tilde{C}_1 (log \lvert \mathcal{G}_{R} (X_l)/G_n \rvert)^{C_2}$. 
\end{center}
The map $\rho \circ \psi : \mathcal{G}_{R} (X_l) \twoheadrightarrow G_{ad}$ descends, by (\ref{exceptionalquotient}), 
to an epimorphism $\mathcal{G}_{R} (X_l) / G_n \twoheadrightarrow G_{ad} / K_n$. 
By Lemma \ref{easyextnlem} (i), 
\begin{center}
$diam(G_{ad} / K_n) \leq diam(\mathcal{G}_{R} (X_l)/G_n)$. 
\end{center}
Finally, $\lvert \mathcal{G}_{R} (X_l)/G_n \rvert \ll_{R,X_l} \lvert R/\mathcal{M} \rvert^{dn}$ 
and $\lvert G_{ad} / K_n \rvert \geq \lvert R/\mathcal{M} \rvert^n$, 
so \\ $(log\lvert \mathcal{G}_{R} (X_l) / G_n \rvert)^{C_2} \ll (log\lvert G_{ad} / K_n \rvert)^{C_2}$ 
and the result follows (replacing $\tilde{C}_1$ by some larger constant $C_1$). 
\\ \\
The bound we thus obtain for $C_2$ is independent of $X_l$, 
since we need only apply Theorem \ref{Rstandarddiam} for finitely many types $X_l$. 
\end{proof}

\begin{rmrk}
\begin{itemize}
\item[(i)] The method of this section is also applicable to the classical Chevalley groups, 
though does not yield uniformity in the exponent $C_2$. 
In particular we obtain a diameter bound in the case \linebreak$(X_l,p)=(B_l,2)$ or $(D_l,2)$, 
which does not fall under the purview of Theorem \ref{chevdiam}. 
The case $(X_l,p)=(A_1,2)$ or $(C_l,2)$ is beyond the scope of our methods, however, 
because the associated Lie algebras are not perfect. 

\item[(ii)] The best degree $C_2$ in Theorem \ref{chevdiam} 
which we can obtain by the above method is based on taking $A=248$ in Proposition \ref{SolovayKitaev2}, 
because $248$ is the dimension of $\mathcal{G}_{R} (X_l)$ as an $R$-analytic group in the case $X_l = E_8$. 
It is likely that this is far from optimal, and that a much lower degree could be obtained 
via a more direct analysis of the Lie algebras of the exceptional groups, 
akin to that employed for the classical groups in Section \ref{classicalsection}. 
In the case $R = \mathbb{Z}_p$, this has already largely been achieved by Dinai in \cite{Dinai}: 
he showed that for $p > 19$, every element of the $\mathbb{Z}_p$-Lie ring 
associated to an exceptional group can be expressed as the sum of three brackets. 
\end{itemize}
\end{rmrk}

\section{The Nottingham Group} \label{NottinghamSection}

We first collect some facts about generation and commutators in $\mathcal{N}_q$ 
with which to deduce Theorem \ref{nottinghamdiam} from Proposition \ref{SolovayKitaev2}. 
Details can be found in \cite{Camina}; \cite{Klopsch}; \cite{York}. 
For $n \geq 2$ and $\lambda \in \mathbb{F}_q$, define:
\begin{center}
$e_{n,\lambda} (t) = t + \lambda t^{n+1} \in K_n$. 
\end{center}

The elements $e_{n,\lambda}$ form an infinite topological generating set for $\mathcal{N}_q$, as follows: 

\begin{lem} \label{NottinghamGen}
\begin{itemize}
\item[(i)] For any $n \geq 1$ $\lambda , \mu \in \mathbb{F}_q$, 
\begin{center}
$e_{n , \lambda} \cdot e_{n , \mu} \equiv e_{n , \lambda+\mu} \text{ mod }K_{2n}$
\end{center}
(so in particular $e_{n , \lambda} ^k \equiv e_{n , k \lambda} \text{ mod }K_{2n}$ for all $k \in \mathbb{N}$). 
\item[(ii)] $\mathcal{N}_q = \lbrace e_{1,\lambda_1} \cdot e_{2,\lambda_2} \cdots : 
(\lambda_k)_k \in \mathbb{F}_q ^{\mathbb{N}} \rbrace$. 
\end{itemize}
\end{lem}

The commutator structure of $\mathcal{N}_q$ is well-behaved; 
in particular we verify hypothesis (i) of Proposition \ref{SolovayKitaev2}: 

\begin{lem} \label{NottinghamComms}
Let $m , n \in \mathbb{N}$. 
\begin{itemize}
\item[(i)] Let $g = t + \sum_{k=n+1} ^{\infty} \lambda_k t^k \in K_n \setminus K_{n+1}$, 
$h = t + \sum_{k=m+1} ^{\infty} \mu_k t^k$ \linebreak$\in K_m \setminus K_{m+1}$, so that 
$\lambda_{n+1}, \mu_{m+1} \neq 0$. Then: 
\begin{center}
$[g,h] \equiv t + \lambda_n \mu_m (n-m) t^{m+n+1}\text{ mod } K_{m+n+1}$. 
\end{center}
\item[(ii)] For any $\lambda , \mu \in \mathbb{F}_q$, 
\begin{center}
$[e_{n,\lambda},e_{m,\mu}] \equiv e_{m+n,\lambda \mu} ^{n-m} \text{ mod } K_{min(m+2n,2m+n)}$. 
\end{center}
\item[(iii)] For $p \geq 3$, if $p \nmid (n-m)$ (respectively $p \mid (n-m)$), 
then \linebreak$[K_n , K_m] = K_{m+n}$ (respectively $[K_n , K_m] = K_{m+n+1}$). 
\end{itemize}
\end{lem}

We shall show that, provided $p \geq 3$, for $n \leq m \leq 2n$ satisfying $p \nmid (m-n)$
every element of $K_{m+n}$ may be expressed, modulo $K_{m+2n}$, 
as $[g_1 , h_1][g_2 , h_2]$ for some $g_i \in K_m$, $h_i \in K_n$. 
Now, for any $\epsilon \in (0,1)$, $n \geq 5$ and $i=1,2,3$, 
there exist $n_i , m_i \in \mathbb{N}$ such that $n_i + m_i =(2+i)n$; 
$\frac{n}{3} (2+i+ \epsilon) \leq n_i \leq m_i \leq \frac{2 n}{3} (2+i)$ and $m_i - n_i \in \lbrace 1,2 \rbrace$. 
We therefore satisfy hypothesis (ii) of Proposition \ref{SolovayKitaev2} 
with $\epsilon$ arbitrary; $A=2$; $M_1=5$; $M_2=0$. 
\\ \\
For any $\lambda_i , \nu \in \mathbb{F}_q$; $K,M,N \in \mathbb{N}$ with $N \leq M$; 
applying Lemma \ref{NottinghamComms} (iii) and an easy induction, we have: 
\begin{center}
$[g,e_{M,\nu}] 
\equiv [e_{N,\lambda_1},e_{M,\nu}] \cdots [e_{N+K-1,\lambda_K},e_{M,\nu}] \text{ mod } K_{2N+M+1}$
\end{center}
where $g = e_{N,\lambda_1} \cdots e_{N+K-1,\lambda_K}$. 
Moreover, by Lemma \ref{NottinghamGen} (i) and Lemma \ref{NottinghamComms} (ii), 
\begin{center}
$[e_{N+i,\lambda_{i+1}},e_{M,\nu}] \equiv (e_{M+N+i,\lambda_{i+1} \nu})^{(N-M)+i} \text{ mod } K_{2N+M+2i} K_{N+2M+i}$
\\
$\equiv e_{M+N+i,\lambda_{i+1} \nu ((N-M)+i)} \text{ mod } K_{2M+2N+2i}$. 
\end{center}
Hence for any $\lambda_i , \mu_i \in \mathbb{F}_q$, setting:
\begin{center}
$g_1 = e_{n,\lambda_1} \cdots e_{2n-1,\lambda_n}$
\\ 
$g_2 = e_{n,\mu_1} \cdots e_{2n-2,\mu_{n-1}}$
\end{center}
we have: 
\begin{center}
$[g_1 , e_{m,1}] [g_2, e_{m+1,1}] \equiv (\prod_{i=0} ^{n-1} e_{n+m+i,\lambda_{i+1}(n+i-m)} )(\prod_{i=1} ^{n-1} e_{n+m+i,\mu_i (n-m-2+i)})$
\\
$\equiv e_{n+m,\lambda_1 (n-m)} (\prod_{i=1} ^{n-1} e_{n+m+1,\lambda_{i+1}(n-m+i)+\mu_i (n-m-2+i)}) \text{ mod } K_{2n+m}$
\end{center}

since $K_{n+m}/K_{2n+m}$ is abelian. 
$p \nmid (n-m)$, and since $p \geq 3$, for each $1 \leq i \leq n-1$, 
$p$ divides at most one of $n-m+i,n-m-2+i$. 
Hence by varying the $\lambda_i$ and $\mu_i$, using the form described in Lemma \ref{NottinghamGen} (ii), 
we can express any element of $K_{n+m}$ modulo $K_{2n+m}$. 

\section{Limit Theorems for Random Walks} \label{RWSection}

The purpose of this section is to prove Corollaries \ref{rwstandard}, \ref{rwpadic} and \ref{rwnottingham}. 
Let $\Gamma$ be a countable group. For $\phi,\psi \in l^2 (\Gamma)$, with $\phi$ of finite support, 
we define the convolution $\phi * \psi \in l^2 (\Gamma)$ by:
\begin{center}
$(\phi * \psi)(g) = \sum_{h \in \Gamma} \phi(h) \psi(h^{-1} g)$. 
\end{center}
For $l \in \mathbb{N}$, we define the \emph{convolution power} $\phi^{* l}$ of $\phi$ recursively by:
\begin{center}
$\phi^{* 0} = \chi_e$; $\phi^{* (l+1)} = \phi^{* l} * \phi$. 
\end{center}
Let $S \subseteq \Gamma$ be a finite symmetric set. 
Let $X_1 , X_2 , \ldots$ be a sequence of independent random variables, each with law:
\begin{center}
$\frac{1}{\lvert S \rvert} \chi_S \in l^2 (\Gamma)$. 
\end{center}
For $l \in \mathbb{N}$, the simple random walk $Y_l = X_1 \cdots X_l$ 
on $(\Gamma,S)$ at time $l$ has law $\frac{1}{\lvert S \rvert^l} \chi_S ^{* l}$. 
We relate the asymptotics of the distributions of the $Y_l$ to diameters of finite groups via the following method: 
\\ \\
For $G$ a finite group, $S \subseteq G$ a symmetric generating set, 
define  a linear operator $A_S : l^2 (G) \rightarrow l^2 (G)$ (called the \emph{adjacency operator}) by:
\begin{center}
$A_S (f) = (\frac{1}{\lvert S \rvert} \chi_S)  * f$. 
\end{center}
Let $l_0 ^2 (G) \leq l^2 (G)$ be the space of functions of mean zero on $G$ 
(that is, the orthogonal complement of the constant functions), 
and note that $l_0 ^2 (G)$ is preserved by $A_S$. 
Let $\rho$ be the norm of $A_S \mid_{l_0 ^2 (G)}$ in the Banach space 
$B(l_0 ^2 (G))$ of bounded linear operators on $l_0 ^2 (G)$. 
We define the \emph{spectral gap} of the pair $(G,S)$ to be the quantity $1 - \rho$. 
As we intimated in the introduction, a large spectral gap implies rapid mixing of the random walk on $(G,S)$. 
Specifically: 

\begin{lem} \label{mixingbound}
For any $l \in \mathbb{N}$; $g , h \in G$, 
\begin{center}
$\lvert \langle A_S ^l \chi_g , \chi_h \rangle - \frac{1}{\lvert G \rvert} \rvert \leq  \rho^l$. 
\end{center}
\end{lem}

\begin{proof}
Noting that $\chi_g - \frac{1}{\lvert G \rvert} \chi_G \in l_0 ^2 (G)$, 
\begin{center}
$\lvert \langle A_S ^l \chi_g , \chi_h \rangle - \frac{1}{\lvert G \rvert} \rvert
 = \lvert \langle A_S ^l (\chi_g - \frac{1}{\lvert G \rvert} \chi_G) , \chi_h \rangle \rvert $
 \\
 $\leq \lVert A_S ^l (\chi_g - \frac{1}{\lvert G \rvert} \chi_G) \rVert_2$
\end{center}
by the Cauchy-Schwarz inequality. 
The result follows, since: 
\begin{center}
$\lVert \chi_g - \frac{1}{\lvert G \rvert} \chi_G \rVert_2 \leq 1$. 
\end{center}
\end{proof}

Finally, $\rho$ is related to $diam(G,S)$ via the following inequality (see \cite{DiSC} for a proof): 

\begin{propn} \label{diamgap}
$\frac{diam(G,S)-1}{log \lvert G \rvert} \leq \frac{1}{1-\rho} \leq \lvert S \rvert diam(G,S)^2$. 
\end{propn}

In particular, for $diam(G,S) \leq C_1 log ^{C_2} \lvert G \rvert$, 
\begin{center}
$1-\rho \geq \frac{1}{\lvert S \rvert C_1 ^2 log^{2 C_2} \lvert G \rvert}$
\end{center}
so, setting $C_3 = \lvert S \rvert C_1 ^2$, and applying Lemma \ref{mixingbound}, we have:
\begin{center}
$\lvert \langle A_S ^l \chi_g , \chi_h \rangle - \frac{1}{\lvert G \rvert} \rvert 
\leq (1 - \frac{1}{C_3 log^{2 C_2} \lvert G \rvert})^l$. 
\end{center}
Recall that $(1-\frac{1}{x})^x$ is an increasing function for $x > 1$, 
converging to $e^{-1}$ as $x \rightarrow \infty$. 
Hence, setting $l = C_3 log^{2 C_2 + C_4} \lvert G \rvert$, for some $C_4 >0$, we deduce:
\begin{center}
$\lvert \langle A_S ^l \chi_g , \chi_h \rangle - \frac{1}{\lvert G \rvert} \rvert 
\leq e^{- log^{C_4} \lvert G \rvert}$. 
\end{center}
Moreover, the quantity $\lvert \langle A_S ^l \chi_g , \chi_h \rangle - \frac{1}{\lvert G \rvert} \rvert$ 
is non-increasing, so this last inequality holds for any $l \geq C_3 log^{2 C_2 + C_4} \lvert G \rvert$. 

\begin{proof}[Proof of Corollary \ref{rwstandard}]
We may identify: 
\begin{center}
$G / K_{N+1} \cong \lbrace \lambda_1 x_1 + \cdots + \lambda_d x_d : 
\lambda_1 , \ldots , \lambda_d \in R / \mathcal{M}^N \rbrace 
\cong (R / \mathcal{M}^N)^d$, 
\end{center}
as a set, so $\lvert G / K_{N+1} \rvert = \lvert R / \mathcal{M} \rvert^{dN}$ and:
\begin{center}
$\Big\lvert \mathbb{P} \big[ \lVert L_1 ^{(l)} - \lambda_1 \rVert , \ldots , \lVert L_d ^{(l)} - \lambda_d \rVert
 \leq c^{N+1} \big] - \frac{1}{\lvert R / \mathcal{M} \rvert^{dN}} \Big\rvert
 = \lvert \langle A_S ^l \chi_e , \chi_g \rangle - \frac{1}{\lvert G / K_{N+1} \rvert} \rvert$
\end{center}
where $g = \lambda_1 x_1 + \cdots + \lambda_d x_d \in G / K_{N+1}$. 
The result is now a consequence of Theorem \ref{Rstandarddiam} 
and the discussion following Proposition \ref{diamgap}, taking:
\begin{center}
$C=2C_2, C'=C_4 , C''=C_3 (d \cdot log \lvert R / \mathcal{M} \rvert)^{2C_2 + C_4} , 
C'''= (d \cdot log \lvert R / \mathcal{M} \rvert)^{C_4}$. 
\end{center}
\end{proof}

\begin{proof}[Proof of Corollary \ref{rwpadic}]
By Theorem \ref{padicfiltgen}, 
\begin{center}
$G / K_{N+1} \cong \langle K_{N+1} a_1 \rangle \times \ldots \times \langle K_{N+1} a_d \rangle 
\cong (\mathbb{Z} / p^N \mathbb{Z})^d$, 
\end{center}
as a set, so $\lvert G / K_{N+1} \rvert = p^{dN}$ and:
\begin{center}
$\Big\lvert \mathbb{P} \big[ \lVert M_1 ^{(l)} - \mu_1 \rVert , \ldots , \lVert M_d ^{(l)} - \mu_d \rVert
 \leq p^{-N-1} \big] - \frac{1}{p^{dN}} \Big\rvert
 = \lvert \langle A_S ^l \chi_e , \chi_g \rangle - \frac{1}{\lvert G / K_{N+1} \rvert} \rvert$
\end{center}
where $g = K_{N+1} a_1 ^{\mu_1} \cdots a_d ^{\mu_d} \in G / K_{N+1}$. 
The result now follows from Theorem \ref{padicdiam} 
and the discussion following Proposition \ref{diamgap}, taking: 
\begin{center}
$C=2C_2, C'=C_4 , C''=C_3 (d \cdot log (p))^{2C_2 + C_4} , C'''= (d \cdot log(p))^{C_4}$. 
\end{center}
\end{proof}

\begin{proof}[Proof of Corollary \ref{rwnottingham}]
Letting $G_N = \mathcal{N}_q / K_N$, $\lvert G_N \rvert = q^{N-1}$, so: 
\begin{center}
$\Big\lvert \mathbb{P} \big[ A_2 ^{(l)} = \alpha_2 , \ldots , A_N ^{(l)} = \alpha_N \big] - \frac{1}{q^{N-1}} \Big\rvert
=\lvert \langle A_S ^l \chi_e , \chi_g \rangle - \frac{1}{\lvert G_N \rvert} \rvert$, 
\end{center}
where $g = t + \sum_{i=2} ^N \alpha_i t^i$. 
The result follows from Theorem \ref{nottinghamdiam} and the discussion following Proposition \ref{diamgap}, taking:
\begin{center}
$C=2C_2, C'=C_4 , C''=C_3 (log(q))^{2C_2 + C_4}  , C'''= {log(q)}^{C_4} $. 
\end{center}
\end{proof}

\end{document}